\documentclass[11pt]{article}
\usepackage{amssymb}
\usepackage{caption}
\usepackage[colorlinks,
            linkcolor=blue,
            anchorcolor=blue,
            citecolor=blue]{hyperref}
\usepackage{url}
\usepackage[dvips]{graphicx}
\usepackage{graphics}
\usepackage[dvips]{color}
\usepackage{xcolor}
\usepackage{multirow}
\usepackage[top=2cm, bottom=2cm, left=2cm, right=2cm]{geometry}
\usepackage{algorithm}
\usepackage{algorithmicx}
\usepackage{algpseudocode}
\usepackage{amsmath}
\usepackage{amsthm}
\usepackage{amsmath}
\usepackage{amsfonts}
\usepackage{epic,eepic}
\usepackage{epsfig}
\usepackage{booktabs} %
\usepackage[toc,page,title,titletoc,header]{appendix}

\usepackage{amssymb}
\usepackage{amssymb}
\usepackage{url}
\usepackage{multirow}
\usepackage{algorithmicx}
\usepackage{algpseudocode}
\usepackage[figuresright]{rotating}
\usepackage{amsfonts}
\usepackage{booktabs}

\newtheorem{thm}{Theorem}[section]
\newtheorem{lem}[thm]{Lemma}

\newtheorem{rem}[thm]{Remark}

\newtheorem{dfn}[thm]{Definition}

\def\ra{{\rm Range}}

\def\v{{\bf v}}

\def\b{{\bf b}}
\def\d{{\bf d}}
\def\e{{\bf e}}
\def\A{{\bf A}}
\def\D{{\bf D}}
\def\I{{\bf I}}
\def\x{{\bf x}}
\def\y{{\bf y}}
\def\z{{\bf z}}
\def\u{{\bf u}}
\def\s{{\bf s}}

\def\M{{\bf M}}

\def\P{{\bf P}}
\def\R{{\mathbb R}}
\def\Null{{\rm Null}}

\def\dt{{\rm dist}}

\DeclareMathOperator*{\argmin}{argmin}

\usepackage{tikz}
\usetikzlibrary{arrows, shapes, decorations.markings}
\usetikzlibrary{calc,angles,quotes}

\usepackage{mathtools}
\newcommand\norm[1]{\left\lVert#1\right\rVert}%
\def\inprod#1#2{\langle#1,\,#2\rangle}
\author{Rujun Jiang\thanks{School of Data Science, Fudan University, Shanghai, China, rjjiang@fudan.edu.cn}
\and Xudong Li\thanks{School of Data Science, and Shanghai Center for Mathematical Sciences, Fudan University, Shanghai, China, lixudong@fudan.edu.cn}
}
\title{H\"olderian error bounds  and Kurdyka-{\L}ojasiewicz inequality for the trust region subproblem}
\begin{document}

\maketitle
\begin{abstract}
In this paper, we study the local variational geometry of the optimal solution set of the trust region subproblem (TRS), which minimizes a general, possibly nonconvex, quadratic function over the unit ball.
Specifically, we demonstrate that a H\"olderian error bound holds globally for the TRS with modulus 1/4 and the Kurdyka-{\L}ojasiewicz (KL) inequality holds locally for the TRS with a KL exponent 3/4 at any optimal solution.
We further prove that unless in a special case, the  H\"olderian error bound modulus, as well as the KL exponent, is 1/2.
Finally, as an application, based on the obtained KL property, we further show that projected gradient methods studied in [A. Beck and Y. Vaisbourd, SIAM
J. Optim., 28 (2018), pp. 1951--1967] for solving the TRS achieve a local sublinear or even linear rate of convergence.
\end{abstract}

\section{Introduction}
In this paper, we consider the following trust region subproblem (TRS)
\begin{eqnarray*}
(\rm P_0)& \min&f(\x):= \x^T\A\x-2\b^T\x\\
&\rm s.t.&g(\x):=\|\x\|^2-1\le0,
\end{eqnarray*}
where $\A$ is an $n\times n$ nonzero real symmetric matrix, $\b\in \R^n$ and
$\|\cdot\|$ denotes the Euclidean $l_2$ norm. To avoid the trivial case, we make the blanket assumption that $n\ge 2$. Let $S_0$ be the nonempty optimal solution set of $\rm(P_0)$ and $f^*$ be the optimal value.

Problem $\rm(P_0)$ first arises as a subproblem in the trust region
 method for nonlinear optimization \cite{conn2000trust,yuan2015recent}, and also admits applications in robust optimization \cite{ben2009robust} and  least squares problems \cite{zhang2010derivative}. The generalized  trust region subproblem, where the constraint is a general quadratic inequality, is also well studied in the literature \cite{more1993generalizations,pong2014generalized,jiang2018socp,jiang2019novel,jiang2018linear}.
When $\A$ is not positive semidefinite, problem $\rm(P_0)$ is a nonconvex problem and may have a local non-global optimal solution \cite{martinez1994local}.
However, problem $\rm(P_0)$ enjoys hidden convexity and the strong duality holds due to the celebrated S-lemma \cite{yakubovich1971s}.
Various methods have been developed in the literature for solving the TRS, e.g., a safeguarded Newton's method by solving the so-called secular equation \cite{more1983computing},  generalized Lanczos methods and its recent variants \cite{gould1999solving,zhang2017generalized,zhang2018nested} and a  parametrized
eigenvalue approach based on semidefinite programming and duality theory \cite{rendl1997semidefinite}, to name a few.
Hazan and Koren \cite{hazan2016linear} proposed the first linear-time algorithm with complexity $\tilde O(\frac{N}{\sqrt{\epsilon}})$ for the TRS to achieve an $\epsilon$ optimal solution, where $N$ is the number of nonzero entries in the input.
Wang and Xia \cite{wang2017linear}  and Ho-Nguyen and Kilinc-Karzan \cite{ho2017second}  presented a linear-time algorithm  to solve the TRS by applying Nesterov's accelerated gradient descent algorithm to a convex reformulation of $\rm(P_0)$, where such a reformulation was first proposed, to the best of our knowledge, by Flippo and Jansen \cite{flippo1996duality}.
Very recently, Beck and Vaisbourd \cite{beck2018globally} showed that a family of first-order methods for solving problem $\rm(P_0)$, including the projected and conditional gradient methods,  converges to a global optimal solution under proper initialization.
However, the convergence rates for these methods are still largely unknown.

Error bounds \cite{luo1993error,pang1997error}  and the Kurdyka-{\L}ojasiewicz (KL) inequality \cite{attouch2010proximal} are widely used in the literature for analyzing rate of convergence for various optimization algorithms and for understanding the local variational geometry of the solution set of an optimization problem.
In this paper, we devote ourself to a thorough understanding of error bounds and the KL inequality for the TRS, which, as far as we know, are not available in the literature. Then, based on these results, we conduct an analysis of the convergence rate of projected gradient methods studied in \cite{beck2018globally}.
{Before stating the defintions of the H\"olderian error bounds and the KL inequality for the TRS, let us first define $B$ as the unit norm ball, i.e., $B := \{ \x\in\R^n \mid \norm{\x} \le 1\}$ and $\delta_{B}$ as the corresponding indicator function}  $\delta_B(\x)=\left\{\begin{array}{ll}0,&\text{if }\|\x\|\le1,\\
+\infty,&\text{otherwise.}\end{array}\right.$
\begin{dfn}[H\"olderian error bounds for the TRS]
        \label{def:errorbound}
        The TRS is said to satisfy the global error bound condition with modulus $\rho$ if there exists a constant $\tau >0$ such that
        \[
        {\rm dist}(\x,S_0)\le \tau \big(f(\x) - f^* + \delta_B(\x)\big)^{\rho}, \quad \forall\, \x\in \R^n.
        \]
\end{dfn}
\noindent Notice that Definition \ref{def:errorbound} is sometimes referred to as the H\"olderian growth condition. Specifically, if $\rho = 1/2$, it is exactly the quadratic growth condition \cite{dontchev2009implicit}; and if $\rho = 1$, the objective $f(\x) + \delta_B(\x)$ is said to have weak sharp minima \cite{burke1993weak}. We shall emphasize that the residual function used in Definition \ref{def:errorbound}, i.e., the right-hand-side of the inequality, corresponds to the difference between the objective function value at any given point and the optimal value of the TRS.
This is different from (sub)gradient based error bounds, e.g., Luo-Tseng error bounds \cite{luo1993error}, where residual functions associated with certain (sub)gradient information are used in the definitions.
\begin{dfn}[KL inequality for the TRS]
        \label{def:KL}
        The TRS is said to satisfy the KL inequality at $\x^*\in S_0$ with an exponent $\varrho\in[0,1)$, if there exist $\tau >0$ and $\epsilon >0$ such that
        \[
        \big(f(\x) - f^* + \delta_{B}(\x)\big)^{\varrho} \le \tau {\rm dist}(-\nabla f(\x), N_B(\x)), \quad \forall \, \x\in B(\x^*,\epsilon),
        \]
        {where $N_B(\x)$ is the normal cone of $B$ at $\x$, $B(\x^*,\epsilon)$ is the $\epsilon$-neighborhood of $\x^*$, i.e., $B(\x^*,\epsilon) = \left\{\x\in \R^n \mid \norm{\x - \x^*}\le \epsilon \right\}$,
        {and by convention the distance of any given point to an empty set is defined as $\infty$}.}
\end{dfn}
\noindent We note that Definition \ref{def:KL} is inherited from the definition given in \cite{bolte2017error}, except  that our definition focuses only on optimal solutions instead of general stationary points.
Existing works of error bounds and the KL inequality related to the TRS are H\"olderian error bounds for convex quadratic inequality systems  \cite{wang1994global}  (see also Subsection 3.1),  convex piecewise quadratic functions \cite{li1995error}, a system consisting of a nonconvex quadratic function and a polyhedron  \cite{luo2000error}, and the KL property for spherical constrained quadratic optimization problems \cite{liu2016quadratic,gao2016ojasiewicz} and spherical constrained quartic-quadratic optimization problems \cite{zhang2019geometric}. However, none of the above results can be directly applied to the TRS due to the existence of the possibly nonconvex and nonhomogeneous quadratic objective function and the unit ball constraint.

Recently, \cite{bolte2017error} studied the relations between H\"olderian error bounds and the KL inequality for convex problems. Specifically, they showed that for convex problems, error bounds with moderate residual functions are equivalent to the KL inequalities and the sum of the H\"olderian error bound modulus and the KL exponent equals one. See also \cite{aragon2008characterization,cui2019r,li2018calculus,drusvyatskiy2018error,drusvyatskiy2014second} for other related discussions and results. {As far as we know}, all the available equivalent relations between H\"olderian error bounds and the KL inequality are obtained under the convexity assumption.
For nonconvex problems or even the special nonconvex TRS, however, such relations remain largely unknown.

In this paper, by combining the local geometry of the optimal solution set $S_0$ and the elegant H\"olderian error bound results for convex quadratic inequality systems in \cite{wang1994global}, we are able to obtain a comprehensive
characterization of  a H\"olderian error bound for the TRS.
Specifically, it can be shown that the H\"olderian error bound  holds globally  with the modulus $\rho = 1/4$
for the \emph{TRS-ill case} (to be defined in \eqref{eq:con}) and otherwise, $\rho = 1/2$.
Then, based on the obtained error bound results, we are able to derive the KL inequality for the TRS.
We show that for the TRS, the KL inequality always holds locally  with the KL exponent $\varrho = 3/4$ at any optimal solution (if the TRS is convex, the KL inequality in fact holds globally).
More precisely, the KL exponent is 3/4 if we are dealing with the \emph{TRS-ill case}  and 1/2 otherwise at any optimal solution, i.e., the sum of the KL exponent $\varrho$ and the H\"olderian error bound modulus $\rho$ always equals one for the TRS. Hence, we successfully extend the equivalence between error bounds  and the KL inequality from the convex problems \cite{bolte2017error}  to the nonconvex TRS.
We shall emphasize here that for the TRS, both error bounds and the KL inequality results, as well as their relations, are new in the literature.
Equipped with these thorough understandings, we are able to derive convergence rate results for algorithms for solving the TRS.
As an illustration, the convergence rate of {projected gradient methods considered in \cite{beck2018globally} is studied.
Specifically, we show that projected gradient methods converge to a global optimal solution locally sublinearly in the \emph{TRS-ill case}  and  linearly otherwise.

The remaining of this paper is organized as follows. In Section 2, we review some existing results in the literature that will be used in our proof. Then, we conduct a thorough analysis about the H\"olderian error bound and the KL inequality for the TRS in Sections 3 and 4, respectively.
In Section 5, we study the convergence rate of projected gradient methods for solving the TRS with the help of the KL inequality.
We conclude our paper in Section 6.

\textbf{Notation.}
For any vector $\x\in\R^n$, we use $[\x]_+$ to denote its positive part. We use $(\cdot)^\dagger$ to denote the \textit{Moore-Penrose} pseudoinverse of a matrix. For any given nonempty closed set $C$, the distance between any given vector $\x\in\R^n$ to $C$ is denoted by ${\rm dist}(\x,C) := \min_{\y\in C} \norm{\x - \y}$. Meanwhile, define the possibly set-valued projection mapping $\Pi_C:\R^n \rightrightarrows \R^n$ as
$
\Pi_{C}(\x) = \left\{\v \in \R^n \mid \norm{\x - \v} = {\rm dist}(\x,C) \right\}.
$
{We define $\bf 0$ as the vector (or matrix) of all zeros and its dimension will be clear from the context.}

\section{Preliminaries}\label{sec:pre}
We recall some basic properties associated with the TRS.
Let $\A = \P{\mathbf \Lambda} \P^T$ be the spectral decomposition of $\A$ with $\P$ being the orthogonal matrix and ${\mathbf \Lambda} = {\rm Diag}(\lambda)$ being the diagonal matrix with {$\lambda_1\le \cdots \le \lambda_n$}.
When $\lambda_1 <0$ (i.e., nonconvex trust region subproblems), we consider the following convex relaxation:
\begin{align*}
(\rm P_1) \quad  \min&\ \tilde f(\x):= \x^T(\A-\lambda_1\I)\x-2\b^T\x+\lambda_1\\
{\rm s.t.}& \ \|\x\|^2\le1.
\end{align*}
Problem ${\rm (P_1)}$ is regarded as a relaxation of ${\rm (P_0)}$  since
\begin{equation}
\label{eq:fandtf}
[g(\x)]_+ = [\norm{\x}^2 - 1]_+ = 0 \quad \mbox{ and } \quad \tilde f(\x)= f(\x)-\lambda_1(\x^T\x-1)\le f(\x)
\end{equation}
whenever $\|\x\|\le1$ and $\lambda_1 <0$.
We shall emphasize here that the relaxation $({\rm P_1})$ plays a central role in our subsequent analysis.
Throughout this paper, we define the solution set of the problem  $\rm(P_1)$  as $S_1$.
We summarize in the following lemma the corresponding results obtained in \cite[Lemmas 1 and 2]{flippo1996duality} to reveal the relations between $S_0$ and $S_1$.
\begin{lem}
        \label{lemma:P1}
        {Suppose $\lambda_1<0$. Then we have the following:
           \begin{itemize}
           \item $\norm{\x^*}=1, \forall \, \x^*\in S_0$,
           \item $\tilde f(\x^*) = f(\x^*) $ for any $\x^*\in S_0$, and
           \item  $S_0 = S_1 \cap \{\x \in \R^n \mid \norm{\x} = 1\}$.
           \end{itemize}}
\end{lem}

Lemma \ref{lemma:P1} also implies the well-known  optimality conditions \cite{flippo1996duality,conn2000trust} for the TRS, i.e.,
$\x^*\in \R^n$ is an optimal solution to problem $\rm(P_0)$ if and only if for some $\lambda^*\in\R$, $(\x^*,\lambda^*)$ satisfies the following KKT conditions:
\begin{eqnarray*}
  \|\x^*\|^2&\le& 1, \\
  (\A+\lambda^*\I)\x^*&=&\b, \\
  \A+\lambda^*\I&\succeq& {\bf 0},\\
  \lambda^*&\ge&0,\\
  \lambda^*(1-\|\x^*\|^2)&=&0~~~ (\text{complementary slackness}).
\end{eqnarray*}

For the noncovnex TRS, i.e., $\lambda_1 <0$, it is  well-known
 that the problem can be categorized into easy and hard cases\footnote{We should point out that in this paper we use the categories of the easy and hard cases only for the nonconvex case, which is slightly different from the categories in \cite{fortin2004trust}.} (\cite{fortin2004trust}).
A brief review about the two cases is given as follows:
\begin{enumerate}
        \item In the easy case, $\b\not\perp\Null(\A-\lambda_1\I)$, which implicitly implies that $\lambda^*>-\lambda_1$. In this case, the optimal solution is unique, and is given by
        $\x^*=(\A+\lambda^*\I)^{-1}\b$ with $\|\x^*\|=1$. %
        \item In the hard case, $\b\perp\Null(\A-\lambda_1\I)$.  In this case, the optimal solution may not be unique. In fact, the optimal solution  is given by either $\x^*=(\A+\lambda^*\I)^{-1}\b$ for some optimal Lagrangian multiplier $\lambda^*>-\lambda_1$ (or called hard case 1 in \cite{fortin2004trust})
        or
        $\x^*=(\A-\lambda_1\I)^{\dagger}\b+\v$, where $\v\in\Null(\A-\lambda_1\I)$ such that $\|\x^*\|=1$ (or called hard case 2 in \cite{fortin2004trust}). Particularly, the case with $\v=\bf0$ is called hard case 2 (i), and the case with with $\v\neq\bf 0$ is called hard case 2 (ii).
\end{enumerate}
We summaries the above characterizations in the following table.

\begin{center}
\begin{tabular}{cccc}
	\toprule
	 Easy case &  Hard case 1 & Hard case 2 (i) & Hard case 2 (ii) \\
	\midrule
	$\b\not\perp\Null(\A-\lambda_1\I)$ & $\b\perp\Null(\A-\lambda_1\I)$ & $\b\perp\Null(\A-\lambda_1\I)$, & $\b\perp\Null(\A-\lambda_1\I)$, \\[2pt]
	(implies $\lambda^* > -\lambda_{1}$) &
	and $\lambda^* > -\lambda_{1}$
	&  $\lambda^* = -\lambda_{1}$ &  $\lambda^* = -\lambda_{1}$\\[2pt]
	& & and $\norm{(\A-\lambda_1\I)^{\dagger}\b} = 1$  & and $\norm{(\A-\lambda_1\I)^{\dagger}\b} < 1$ \\[2pt]
	\bottomrule
\end{tabular}
\captionof{table}{Different cases for the nonconvex TRS, i.e., $\lambda_1 < 0$.}
\end{center}

\section{H\"olderian error bounds for the TRS}
In this section, we present our main results on H\"olderian error bounds for the TRS. Mainly, our analysis will be divided into two cases.
First, we consider the convex case, i.e., the case with $\lambda_1 \ge 0$. The more challenging nonconvex case with $\lambda_1 <0$ will be discussed later.

\subsection{Case with $\lambda_1\ge0$}
In this case, problem ${\rm (P_0)}$ is convex and H\"olderian error bounds for the TRS can be obtained by applying the elegant error bound results derived in \cite{wang1994global} for convex quadratic inequalities.
Let $[m]:=\{1,2,\ldots,m\}$. We recall the main definitions and results in \cite{wang1994global}.
\begin{dfn}
        \label{dfn:singular}
        Consider the inequality system
        \[q_i(\x)\le0, \quad \forall i\in[m].\]
        An inequality $q_i(\x) \le0$ in the system is said to be singular if $q_i(\x) = 0$ for any solution $\x$ to the system. Thus an inequality $q_i(\x) \le 0$ is nonsingular if there is a solution $\x_i$ to the system such that $q_i(\x_i)<0$. If every inequality in the system is singular, we say that the inequality system is singular.
\end{dfn}
\begin{dfn}
        \label{def:critical}
        Let S be a singular system of inequalities. We say that S is critical, if either
        one of the following two conditions holds:
        \begin{enumerate}
                \item at most one of the inequalities is nonlinear; or,
                \item after any one of the nonlinear inequalities is deleted, all the remaining nonlinear
                inequalities become nonsingular.
        \end{enumerate}
\end{dfn}

\begin{dfn}
        \label{def:irregular}
        An inequality in a system is called irregular if it is nonlinear, singular, and
        contained in no critical subsystem.
\end{dfn}
The following definition defines a concept of the degree of singularity, which will be used to determine the modulus of H\"olderian error bounds for a convex quadratic inequality system.
\begin{dfn}
        \label{def:degreeofsingularity}
        Let
        \[q_i(\x)\le0,\quad \forall i\in[m].\]
        be a system of inequalities. If there is no nonlinear, singular inequality, we define the
        degree of singularity of this system to be zero. If there is at least one such inequality, we
        define the degree of singularity of the system to be one plus the number of irregular
        inequalities.
\end{dfn}
The main technical result we will use is the following global error bound for convex quadratic inequality systems.
\begin{lem}[Theorem 3.1 in \cite{wang1994global}]
        \label{lem:eb}
        Suppose
        \[q_i(\x)\le0, \quad \forall i\in[m]\]
        is a convex quadratic system with a nonempty solution set $S$ and let $[m]=K\bigoplus J$,
        where  $K$ is the index set of all the nonsingular constraints and $J$ is the index set of all the singular constraints.
        Then there exists a constant $\tau>0$ such that,
        \begin{equation}
        \label{eq:degree}
        \dt(\x,S)\le \tau\left(\sum_{i=1}^m [q_i(\x)]_++\sum_{j\in J}[q_j(\x)]_+ ^{1/2^d}\right),{\quad \forall \x\in \R^n.}
        \end{equation}
        where $d$ is the degree of singularity of the system\footnote{{The original inequality in \cite[Theorem 3.1]{wang1994global} is $\dt(\x,S)\le \tau(\|[q_K(\x)]_+\|+\|[q_J(\x)]_+\|+\|[q_J(\x)]_+\|^{1/2^d})$, from which inequality \eqref{eq:degree} follows directly with a possibly rescaling of the constant $\tau$.} }.
\end{lem}

We note that one important feature of Lemma 3.5 is that the exponent of the term $[q_J(\x)]_+$ in the above inequality is related to $d$, the degree of singularity of the system.
As one can observe later, it is this special and computable quantity that makes our analysis possible. Indeed, when Lemma \ref{lem:eb} is applied to system
\begin{equation}
        \label{eq:sysorigin}
        f(\x)-f^*\le0, \quad g(\x)\le0,
\end{equation}
        the main task will be computing its degree of singularity.
We first consider a case that the system is minimal, i.e., deleting either inequality yields the system nonsingular.

\begin{lem}
\label{lem:EBnonnegreg}
Assume that $\lambda_1 \ge 0$ and $\min_{\x\in\R^n} f(\x)<f^*$. Then there exists a constant $\tau>0$ such that
\begin{equation}
\label{EB:regular}
\dt(\x,S_0)\le \tau \big(f(\x)-f^*+\delta_{B}(\x)\big)^{1/2}, \quad \forall \x\in \R^n.
\end{equation}
\end{lem}
\begin{proof}
From the definition of $f^*$, we know that the solution set of  system \eqref{eq:sysorigin} is $S_0$. From Lemma \ref{lemma:P1}, it further holds that system \eqref{eq:sysorigin} is singular.
Moreover, we see that when either inequality in \eqref{eq:sysorigin} is deleted, the remaining inequality is nonsingular.
Hence, from Definition \ref{def:critical}, we know that system \eqref{eq:sysorigin} is critical.
Therefore, there is no irregular inequality in \eqref{eq:sysorigin} and the degree of singularity $d$ of system \eqref{eq:sysorigin} equals to $1$.
Hence, by using Lemma \ref{lem:eb}, we have that for all $\x\in \R^n$,
\begin{equation}
\label{eq:distS}
\dt(\x,S_0)\le \tau_1\left([f(\x)-f^*]_++[g(\x)]_+
+[f(\x)-f^*]_+^{1/2}+[g(\x)]_+^{1/2}\right).
\end{equation}
By Weierstrass theorem, we know that $[f(\x) - f^*]_+$ is upper bounded over the unit ball, i.e., there exists a constant $M>0$ such that
\[
M = \max \left\{
[f(\x) - f^*]_+ \mid \norm{\x} \le 1
\right\}.
\]
If $M\le 1$, inequality \eqref{eq:distS} implies that \eqref{EB:regular} holds with $\tau = 2\tau_1$. If $M > 1$, we have from \eqref{eq:distS} that
\[
\dt(\x,S_0)\le 2\tau_1 (f(\x)-f^*+\delta_{B}(\x))^{1/2}, \quad \forall\, \x \mbox{ with } [f(\x) - f^*]_+\le 1,
\]
and
\begin{align*}
\dt(\x,S_0)\le{}& \tau_1 (M + (f(\x)-f^*+\delta_{B}(\x))^{1/2}),\\
\le{}& \tau_1(M+1)(f(\x)-f^*+\delta_{B}(\x))^{1/2},
 \quad \forall\, \x \mbox{ with } [f(\x) - f^*]_+> 1.
\end{align*}
Combining all the above discussions, we see that \eqref{EB:regular} holds with $\tau = \max\{2, M+1\}\tau_1$.
\end{proof}

Now consider the case that $\min_{\x\in\R^n} f(\x) = f^*$. In this case, it is easy to see that $\{f(\x)-f^*\le0\}$ is a  singular system. Hence, the degree of singularity
of system \eqref{eq:sysorigin} and the corresponding error bound modulus depend on
the singularity of the second inequality.
Indeed, since $\min_{\x\in\R^n} f(\x) = f^* > -\infty$, we know that $\b\in \ra(\A)$ and for all $\x \in\R^n$,
\begin{equation*}
\label{eq:fval}
f(\x) - f^* = (\x - \tilde \x)^T \A (\x - \tilde \x),
\end{equation*}
where $\tilde \x=\A^\dagger \b$ and $\norm{\tilde \x} \le 1$.
Therefore,  the optimal solution set of problem (P) can be written as \begin{equation}
\label{eq:s0convex}
S_0= \{\x\in\R^n \mid  \x =  \tilde \x + \d, \, \|\x\|\le1, \, \d \in \Null(\A) \}.
          \end{equation}
With these discussions in hand, we are ready to derive a H\"olderian error bound  in the following lemma.
\begin{lem}
\label{lem:EBnonneg}
Assume that $\lambda_1\ge0$ and $\min_{\x\in\R^n} f(\x) = f^*$. Then, it holds that
\begin{enumerate}
        \item if $\lambda_1 >0$, then
        \[
        \dt(\x,S_0)\le \sqrt{\frac{1}{\lambda_1}}\,[f(\x) - f^* + \delta_B(\x)]^{1/2}, \quad \forall\, \x\in\R^n;
        \]
        \item if $\lambda_1 = 0$ and  $\norm{\tilde \x}= 1$, then $S_0 = \{ \tilde \x\}$ and there exists a constant $\tau >0$ such that
        \begin{equation*}
        \label{EB:1o4}
                \dt(\x,S_0)\le \tau[f(\x) - f^* + \delta_B(\x)]^{1/4}, \quad \forall\, \x\in\R^n;
        \end{equation*}
                \item if $\lambda_1 = 0$ and  $\norm{\tilde \x} <      1$, then there exists a constant $\tau >0$ such that
        \begin{equation*}
        \label{EB:1o2}
        \dt(\x,S_0)\le \tau[f(\x) - f^* + \delta_B(\x)]^{1/2}, \quad \forall\, \x\in\R^n.
        \end{equation*}
\end{enumerate}
\end{lem}
\begin{proof}
Case 1 follows directly from the fact that $f(\x) - f^* = (\x - \tilde \x)^T \A (\x - \tilde \x)\ge\lambda_1\|\x-\tilde\x\|^2$ and $S_0=\{\tilde\x\}$.

For case 2, since $\lambda_1 =0$ and $\{\x\in \R^n \mid \norm{\x} < 1\}\cap S_0 = \emptyset$, from \eqref{eq:s0convex}, we know that $S_0 = \{\tilde \x\}$, i.e., $S_0$ is a singleton.
Thus, the second inequality $g(\x) \le 0$ in system \eqref{eq:sysorigin} is singular. Meanwhile, the assumption that $\min_{\x\in\R^n} f(\x) = f^*$ implies that the only critical subsystem of \eqref{eq:sysorigin} is the first inequality $f(\x) - f^*\le 0$. Therefore, $g(\x) \le 0$ is irregular and the degree of singularity of \eqref{eq:sysorigin} equals to $2$.
Then, Lemma \ref{lem:eb} asserts that there exists a constant $\tau_1>0$ such that
\[\dt(\x,S_0)\le \tau_1 \left([f(\x)-f^*]_++[g(\x)]_++[f(\x)-f^*]_+^{1/4}+[g(\x)]_+^{1/4}\right)\]
for all $\x\in\R^n$.
Following the same arguments in the proof of Lemma \ref{lem:EBnonnegreg}, we know that there exists a constant $\tau >0$ such that
\[
\dt(\x,S_0)\le \tau\big(f(\x) - f^* + \delta_B(\x)\big)^{1/4}, \quad \forall\, \x\in\R^n.
\]

In case 3, we have $\{\x \in \R^n \mid \norm{\x} < 1\}\cap S_0 \neq \emptyset$ and  thus the second inequality $\norm{\x}^2 - 1\le 0$ is nonsingular. Hence the degree of singularity of system \eqref{eq:sysorigin} equals to 1 as there are no irregular inequalities.
Then, Lemma \ref{lem:eb} asserts that there exists a constant $\tau_2>0$ such that
\[\dt(\x,S_0)\le \tau_2 \left([f(\x)-f^*]_++[g(\x)]_++[f(\x)-f^*]_+^{1/2}+[g(\x)]_+^{1/2}\right)\]
for all $\x\in\R^n$.
Similar as the proof of Lemma \ref{lem:EBnonnegreg}, we have
\[
\dt(\x,S_0)\le \tau\big(f(\x) - f^* + \delta_B(\x)\big)^{1/2}, \quad \forall\, \x\in\R^n.
\]
This completes the proof for the lemma.
\end{proof}
\subsection{Case with $\lambda_1<0$}
In this section, we turn our interests to the nonconvex TRS. As is stated in the preliminary section, the nonconvex TRS includes the easy and hard cases. We will derive error bounds for them separately.

Before diving into the proofs, we shall discuss the main ideas here.
Since $\lambda_1 <0$, the first inequality in the quadratic inequality system \eqref{eq:sysorigin} is {nonconvex}.
        Fortunately, the following quadratic inequality system
        \begin{equation}
        \label{eq:sysnew}
        \tilde f(\x)-f^*\le0, \quad g(\x)\le0,
        \end{equation}
        derived from the reformulation $\rm(P_1)$ (see Lemma \ref{lemma:P1}), is always a convex one.
        Hence, we can apply Lemma \ref{lem:eb} to system \eqref{eq:sysnew} and then use the relations between the solution sets of
        systems \eqref{eq:sysorigin} and \eqref{eq:sysnew} to establish meaningful error bounds for nonconvex quadratic inequality system \eqref{eq:sysorigin}.

We first study the easy case and hard case 1. In these two cases, the solution for problem (P) is unique and has unit norm. %

\begin{lem}
\label{lem:ezh1}
In the easy case or hard case 1, there exists some constant $\tau>0$ such that
\begin{equation*}
\dt(\x,S_0)\le \tau\big(f(\x)-f^*+\delta_B(\x)\big)^{1/2},\quad \forall \x\in\R^n.
\end{equation*}
\end{lem}
\begin{proof}
In both the easy case and hard case 1, we see that $\min_{\x\in\R^n} \tilde f(\x) < f^*$. %
Indeed, for the easy case, since $\b\not\perp \Null(\A - \lambda_1 \I)$, we have  $\b\notin\ra(\A-\lambda_1\I)$. Then, it holds that $\min \tilde f(\x)=-\infty<f^*$.
For hard case 1, the optimal solution for  $\min_{\x\in\R^n} \tilde f(\x)$
is achieved by $(\A-\lambda_1\I)^\dagger \b$, whose norm is larger than 1. This gives  $\min_{\x\in\R^n} \tilde f(\x)<f^*$.
Similar to the case studied in Lemma \ref{lem:EBnonnegreg}, the degree of singularity of system \eqref{eq:sysnew} equals to $1$. Hence, there exists some constant $\tau>0$ such that for all
$\x\in\R^n$,
\begin{equation*}
\dt(\x,S_1)\le \tau\big(\tilde f(\x)-f^*+\delta_B(\x)\big)^{1/2}.
\end{equation*}
From \eqref{eq:fandtf},
we see that
\[
\dt(\x, S_1) \le \tau\big( f(\x)-f^*+\delta_B(\x)\big)^{1/2}, \quad \forall \, \x\in\R^n.
\]
Note that in both the easy case and hard case 1, since $\min_{\x\in\R^n} \tilde f(\x) < f^*$,
we have that $\norm{\x} = 1$ for all $\x\in S_1$. This further implies that
$S_1 = S_0$
(in fact, $S_1=S_0=\{(\A-\lambda^* \I)^\dagger \b\}$ for a unique $\lambda^*>-\lambda_1$ such that $\|(\A-\lambda^* \I)^\dagger \b\|=1$) and thus completes the proof.
\end{proof}

In hard case 2, $\b$ is orthogonal to the eigenspace of matrix $\A$ corresponding to the smallest eigenvalue, i.e., $\b\perp \Null(\A-\lambda_1\I)$ and $f^* = \min_{\x\in\R^n} \tilde f(\x)$.
Denote $\bar \x=(\A-\lambda_1\I)^\dagger \b$. Then, we have
\begin{equation}
\label{eq:s0s1}
\begin{array}{lll}
S_0&=&\{\bar \x+ \v \in \R^n \mid  \|\bar \x+ \v\|=1,\v\in \Null(\A-\lambda_1\I)\}~~~~ \mbox{
and }\\
S_1&=&\{\bar \x+ \v \in \R^n \mid \|\bar \x+ \v\|\le1,\v\in \Null(\A-\lambda_1\I)\}.
\end{array}
\end{equation}
Let us first consider the  hard case 2 (i).
\begin{lem}
        \label{lem:h2i}
         In the hard case 2 (i),
         there exists $\tau >0$ such that
        \[
        \dt(x,S_0) \le \tau \big(f(\x) - f^* + \delta_{B}(x)\big)^{1/4}, \quad \forall\,
        \x\in\R^n.
        \]
\end{lem}
\begin{proof}
        From the assumption, we know that $S_0 = S_1 = \{\bar \x\}$. By applying Lemma \ref{lem:EBnonneg} case 2 to problem $({\rm P_1})$, we obtain that\[
        \dt(\x,S_0) = \dt(\x,S_1)\le \tau\big(\tilde f(\x)-f^*+\delta_B(\x)\big)^{1/4}, \quad \forall\, \x\in\R^n.
        \]
        The conclusion then follows directly from \eqref{eq:fandtf}.
\end{proof}

Next we focus on the  hard case 2 (ii), in which $S_1$ is not a singleton since $\norm{\bar \x} < 1$.
To establish the desired error bound inequality in this case, we need to characterize the connection between $\dt(\x,S_0)$ and $\dt(\x,S_1)$. For this purpose, we establish the following technical lemma.

\begin{lem}
\label{lem:proj0}
Given $\x\in\R^n$, let  $\x_1:=\Pi_{S_1}(\x)$ be the projection of $\x$ onto the convex set $S_1$.
In the hard case 2 (ii), it holds that
\begin{enumerate}
\item if $\x_1 = \bar \x$, then $\Pi_{S_0}(\x) = S_0$ and
$\dt(\x, S_0) = \sqrt{\norm{\x - \bar \x}^2 + 1 - \norm{\bar \x}^2}\,$;
\item else if  $\x_1 \neq \bar \x$, then
$\Pi_{S_0}(\x) = \bar\x+t_0\v_1$ with $\v_1=\frac{\x_1-\bar\x}{\|\x_1-\bar\x\|} \in \Null(\A - \lambda_1 \I)$  and  $t_0\ge\|\x_1-\bar\x\|$   such that $\|\bar\x+t_0\v_1\|=1.$
\end{enumerate}
\end{lem}
\begin{proof}
In the first case, since $\x_1 = \Pi_{S_1}(\x) = \bar \x $, we know that
\[
\inprod{\x - \bar\x}{\z - \bar\x} \le 0, \quad \forall\, \z \in S_1,
\]
which, together with the structure of $S_1$, implies that
$\x - \bar \x \in \ra(\A - \lambda_1 \I).$
Thus, for any $\y \in S_0$, we have that
\[
\norm{\x - \y}^2 = \norm{\x - \bar \x + \bar \x - \y}^2
= \norm{\x - \bar \x}^2 + \norm{\bar\x - \y}^2,
\]
where the second equality follows from the facts that $\y - \bar\x \in \Null(\A - \lambda_1 \I)$ and {thus} $\inprod{\x - \bar \x}{\bar\x - \y} = 0$.
This, together with $1 = \norm{\y} = \norm{\bar \x - \y}^2 + \norm{\bar \x}^2$ (due to $\y - \bar\x \in \Null(\A - \lambda_1 \I)$ and $\bar\x \in \ra(\A - \lambda_1\I)$), implies
\[\norm{\x - \y} = \sqrt{\norm{\x - \bar \x}^2 + 1 - \norm{\bar \x}^2},\quad \, \forall\, \y \in S_0.\] This completes the proof for the first case.

We next consider the second case.
If $\norm{\x_1} = 1$, we know that $\x_1 \in \Pi_{S_0}(\x)$. For any $\tilde \x \in \Pi_{S_0}(\x)$, it holds that
\[\tilde \x \in S_0 \subseteq S_1 \quad \mbox{ and } \quad {\rm dist}(\x, S_1) = \norm{\x - \x_1} = {\rm dist}(\x, S_0)  = \norm{\x - \tilde \x},\]
and consequently, $\tilde \x = \Pi_{S_1}(\x) = \x_1$ {due to the uniqueness of the projection onto the convex set $S_1$}. %
Hence, $\Pi_{S_0}(\x)$ is a singleton, i.e., \[\x_0 := \Pi_{S_0}(\x) = \x_1= \bar \x + \norm{\x_1 - \bar \x}\frac{\x_1 - \bar \x}{\norm{\x_1 - \bar\x}}.\]
Now we consider the case with $\norm{\x_1}<1$.
Without loss of generality, assume that the null space of $\A-\lambda_1\I$ is spanned by
an orthogonal basis $\{\v_1,\v_2,\ldots,\v_k\}$ with some $k\ge 1$, $\v_1$ being the non-zero vector $\frac{\x_1-\bar\x}{\|\x_1-\bar\x\|}$ and all other $\v_i$ being unit-norm vectors.
Then, we can rewrite the solution set of $\rm(P_0)$ as $S_0=\{\x\in\R^n \mid \x = \bar\x+\sum_{i=1}^k\alpha_i\v_i, \, \|\x\|=1\}$.
Since $\x_1 = \Pi_{S_1}(\x)$, we have
\[
\inprod{\x - \x_1}{\z - \x_1} \le 0,\quad \forall\, \z \in S_1.
\]
This, together with \eqref{eq:s0s1} and  $\norm{\x_1}<1$, implies that
$\inprod{\x - \x_1}{\d} = 0$ for all $\d\in \Null(\A - \lambda_1 \I)$. Therefore, there exists some vector $\s\in \ra(\A - \lambda_1 \I)$ such that
$\x = \x_1 + \s$.
 Consider the projection of $\x$ onto $S_0$:
\begin{equation}
\label{prob:projxs0}
\min \|\x-\z\|^2~~~{\rm s.t.}~~\z\in S_0,
\end{equation}
which, due to $\x=\bar\x+\|\x_1-\bar\x\|\v_1+\s$ and $\z=\bar\x+\sum_{i=1}^k\mu_i \v_i$, is equivalent to
\begin{equation}\label{pb:310}
\min_{\mu_1,\ldots, \mu_k} \left\|\|\x_1-\bar\x\|\v_1+\s-\sum_{i=1}^k\mu_i \v_i\right\|^2~~~
{\rm s.t.}~~\left\|\bar\x+\sum_{i=1}^k\mu_i\v_i\right\|=1.
\end{equation}
Since $\s,\v_1,\ldots,\v_k$ are orthogonal to each other and $\norm{\v_i} = 1$ for $i=1,\ldots,k$, the objective in \eqref{pb:310} can be further written as:
\[
\left\|\|\x_1-\bar\x\|\v_1+\s-\sum_{i=1}^k\mu_i \v_i\right\|^2 = \|\x_1-\bar\x\|^2-2\mu_1\|\x_1-\bar\x\|+
\sum_{i=1}^{k} \mu_i^2
+\|\s\|^2.
\]
We further note that for any feasible solution $(\mu_1,\ldots,\mu_k)$ to \eqref{pb:310}, it holds that
$
\norm{\bar\x}^2 + \sum_{i=1}^{k}\mu_i^2 = 1.
$
Hence, \eqref{pb:310}, and consequently \eqref{prob:projxs0}, can be equivalently rewritten as
\begin{equation*}
\min_{\mu_1,\ldots, \mu_k}\|\x_1-\bar\x\|^2-2\mu_1\|\x_1-\bar\x\|+1-\|\bar\x\|^2+\|\s\|^2~~~
{\rm s.t.}~~\norm{\bar\x}^2 + \sum_{i=1}^{k}\mu_i^2 = 1,
\end{equation*}
whose unique optimal solution is clearly  $\mu_1^*=\sqrt{1-\|\bar\x\|^2}$ and $\mu_i^*=0$ for $i=2,\ldots,k$.
Therefore,  the projection problem
\eqref{prob:projxs0} has a unique optimal solution that takes the form
$\z^*=\bar\x+t_0\v_1$ with $t_0 > \norm{\x_1 - \bar \x}$ such that $\norm{\z^*} = 1$.
\end{proof}

{Now we are ready to present the connection between $\dt(\x,S_0)$ and $\dt(\x,S_1)$ in the hard case 2 (ii).}
\begin{lem}
        \label{lemma:xS}
        In the hard case 2 (ii), there exists some constant $\gamma\in(0,\pi/2)$ with $\sin\gamma=\frac{1-\|\bar\x\|}{ \sqrt{(1 - \norm{\bar\x})^2 + 1-\|\bar\x\|^2}}$ such that
        \begin{equation}
        \label{eq:distxx0xx1}
        {\rm dist}(\x,S_1) + \sqrt{1 - \norm{\x}^2}  \ge {\rm dist}(\x, S_0)\sin\gamma , \quad \forall \x \in \left\{ \x\in\R^n \mid \norm{\x}\le 1 \right\}.
        \end{equation}

\end{lem}

\begin{proof}
For any given $\x\in \R^n$ with $\norm{\x}\le 1$ and any given $\x_0 \in \Pi_{S_0}(\x)$, define $\v = \x_0 - \bar\x$.
Since  $\norm{\bar\x} < 1$ in the hard case 2 (ii) and $\norm{\x_0} = 1$, we have that $\norm{\v} \neq 0$.
Define line segment $L:= \left\{\bar\x + \alpha \v \in \R^n \mid \alpha \in [0,1]\right\}$. We note from \eqref{eq:s0s1} and $\bar\x = (\A-\lambda_1\I)^\dagger \b$ that $\inprod{\v}{\bar\x} = 0$.	
If $\bar\x = \Pi_{S_1}(\x)$, then $\Pi_{S_1}(\x)\in L$;
otherwise, by case 2 in Lemma \ref{lem:proj0}, we know that $\x_0 = \Pi_{S_0}(\x) = \bar\x + t (\Pi_{S_1}(\x) - \bar\x)$ for some constant $t \ge 1$. Hence, $\Pi_{S_1}(\x) - \bar\x = \v/t$, i.e., $\Pi_{S_1}(\x) \in L$.
Since $L\subseteq S_1$, it further holds that
\begin{equation}
\label{eq:piLeqpiS}
\Pi_L(\x) = \Pi_{S_1}(\x).
\end{equation}
Let $\Pi_{L}(\x) = \bar\x + \alpha^*\v$. By the definition of $L$ and the properties of the projection operator $\Pi_L$,
it is not difficult to see that
\begin{equation*}
\inprod{\x - \Pi_{L}(\x)}{\v} \left\{
\begin{aligned}
& = 0, \mbox{ if } \alpha^*\in (0,1),\\
& \le 0, \mbox{ if } \alpha^* = 0,\\
& \ge 0, \mbox{ if } \alpha^* = 1.
\end{aligned}
\right.
\end{equation*}
We first
consider the case where $\inprod{\x - \Pi_{L}(\x)}{\v}>0$. In this case, we  have $\alpha^*=1$ and thus
\[
{\rm dist}(\x,S_1) = \norm{\x - \Pi_{S_1}(\x)} = \norm{\x - \Pi_{L}(\x)} = \norm{\x-(\bar\x + \v)}
= \norm{\x - \x_0} = {\rm dist}(\x,S_0),
\]
i.e., \eqref{eq:distxx0xx1} holds trivially.
Then, we
argue that  $\inprod{\x - \Pi_{L}(\x)}{\v} < 0$ cannot occur. Indeed, if this is not the case,
we must have
 $\alpha^* =0$ and  $\inprod{\x - \Pi_{L}(\x)}{\v} < 0$ and thus
 $ \inprod{\x - \bar\x}{\v} < 0$.
Hence we further have that
\begin{equation}
\label{eq:L311}
\begin{array}{lll}
\norm{\x - (\bar\x - \v)}^2 &={}& \norm{\x - \bar\x}^2 + 2\inprod{\x - \bar\x}{\v}
+\norm{\v}^2 \\
&<{}& \norm{\x - \bar\x}^2 - 2\inprod{\x - \bar\x}{\v}
+\norm{\v}^2 \\
&={}& \norm{\x - (\bar\x +\v)}^2 = \norm{\x - \x_0}^2.
\end{array}
\end{equation}
Since {$\inprod{\bar\x}{\v} = 0$}, it holds that
$\norm{\bar\x - \v}^2 = \norm{\bar\x}^2 + \norm{\v}^2 = \norm{\bar\x + \v}^2 = 1$. This, together with the definition of $S_0$ in \eqref{eq:s0s1}, implies that $\bar\x - \v \in S_0$.
Since $\x_0 \in \Pi_{S_0}(\x)$, it holds that $\norm{\x - \x_0} \le \norm{\x - (\bar\x - \v)}$, which contradicts \eqref{eq:L311}.

Thus, in the subsequent analysis, we only need to focus on the case where $\inprod{\x - \Pi_{L}(\x)}{\v} = 0$.

{We first consider the case where $\bar \x=\bf0$. In this case, $\v=\x_0$.
Let $\u=\x-\Pi_L(\x)=\x-\alpha^*\v$. Hence we have $\u\perp\v$.
It holds that
$
\dt(\x, S_1)=\|\u\|, $
and
\[ \dt(\x,S_0) \le \|\x-\x_0\|=\|\u-(1-\alpha^*)\v\|.
\]
We then claim that \eqref{eq:distxx0xx1} holds.
To see this, since
\[
\dt(\x,S_1)^2+1-\|\x\|^2 \le \left(\dt(\x,S_1)+\sqrt{1-\|\x\|^2}\right)^2,
\]
it suffices to show
\[
\dt(\x,S_0)^2\sin\gamma^2 \le\dt(\x,S_1)^2+1-\|\x\|^2, \]
which
is equivalent to
\begin{equation}
\label{eq:equi01}
\|\u+(1-\alpha^*)\v\|^2\le 2\left(\|\u\|^2+1-\|\x\|^2\right)
\end{equation}
due to $\sin\gamma=1/\sqrt2$ (since $\bar\x=0$), $\x=\u+\alpha^*\v$ and $\u\perp\v$. Meanwhile,
\eqref{eq:equi01} is further equivalent to
\[
\|\u\|^2+ (1-\alpha^*)^2\|\v\|^2+2\|\alpha^*\v\|^2\le 2,
\]
which is trivial since $\|\x\|^2=\|\u\|^2+\|\alpha^*\v\|^2\le1$, $0\le\alpha^*\le1$ and $\|\v\|=1$.}

In the following of this proof, we consider the remaining case where $\bar\x \neq \bf0$.
In this case, we derive a lower bound of $\norm{\x - \Pi_{S_1}(\x)} $,{ which is equivalent to $\norm{\x - \Pi_{L}(\x)}$ due to \eqref{eq:piLeqpiS},} by considering
the following optimization problem:
        \begin{equation}
        \label{prob:zy1y0}
        \min_{\z \in \R^n} \left\{ \norm{\z - \Pi_{L}(\z)} \mid \inprod{\z - \Pi_{L}(\z)}{\v} = 0, \, \norm{\z} = \norm{\x}, \, \norm{\z - \x_0} = {\rm dist}(\x,S_0)
        \right\}.
        \end{equation}
From \eqref{eq:piLeqpiS} and $\inprod{\x - \Pi_{L}(\x)}{\v} = 0$, we see that problem \eqref{prob:zy1y0} has a non-empty closed and bounded feasible set as $\x$ is always a feasible solution. Since the objective is continuous in $\z$, by Weierstrass theorem we know that problem \eqref{prob:zy1y0} has a non-empty and compact solution set.
Let $\z^*$ be any optimal solution to problem \eqref{prob:zy1y0}. We know that
${\rm dist}(\x, S_1) =\norm{\x - \Pi_{L}(\x)}\ge \norm{\z^* - \Pi_{L}(\z^*)}$ as $\x$ is a feasible solution to \eqref{prob:zy1y0}.
Since $\|\x\|=\|\z^*\|$,  it further holds that
\begin{equation}
\label{eq:distrelation}
{\rm dist}(\x,S_1) + \sqrt{1 - \norm{\x}^2} \ge \norm{\z^* - \Pi_{L}(\z^*)} + \sqrt{1 - \norm{\z^*}^2}.
\end{equation}

        Recall that $\bar\x\neq\bf0$ and $\v\neq\bf0$. For any $\z \in \R^n$, since $\v\perp\bar \x$, we note that there exist $\lambda, \mu \in\R$ and $\u\in\R^n$  satisfying  $\inprod{\u}{\bar \x} = \inprod{\u}{\v} = 0$ {(note that $\u=0$ if $n=2$)} such that
        \begin{equation*}
        \label{eq:zdecomp}
        \z = \lambda \bar\x + \mu \v + \u.
        \end{equation*}
        Given the structure of $L$, we know that $\Pi_{L}(\z) = \bar \x + \alpha \v$ for some $\alpha \in [0,1]$.
        Now, if, in addition, $\inprod{\z - \Pi_{L}(\z)}{\v} = 0$, it then holds that $\mu = \alpha \in [0,1]$, and
        {since $\x_0 - \Pi_L(\z)=(\bar\x+\v)-(\bar\x+\mu\v) =(1-\mu)\v $, } we further know that
        $\inprod{\z -\Pi_L(\z)}{\x_0 - \Pi_L(\z)} = 0$.
        Hence, for any feasible solution $\z$ to problem \eqref{prob:zy1y0}, it holds that
        \begin{equation}
        \label{eq:dist}
        \norm{\z - \Pi_{L}(\z)}^2
        = \norm{\z - \x_0}^2 - \norm{\x_0 - \Pi_{L}(\z)}^2
        = {\rm dist}^2(\x, S_0) - (1-\mu)^2\norm{\v}^2.
        \end{equation}
        Therefore, problem \eqref{prob:zy1y0} can be equivalently reformulated as
        \begin{equation}
        \label{prob:mu}
        \begin{aligned}
        \min_{\u\in\R^n, \, \mu,\lambda\in\R} \quad &\mu \\
        {\rm s.t.}\quad & \lambda^2 \norm{\bar \x}^2 + \mu^2 \norm{\v}^2 + \norm{\u}^2 = \norm{\x}^2, \\
        & (\lambda -1)^2 \norm{\bar \x}^2 + (\mu -1)^2\norm{\v}^2 + \norm{\u}^2 = {\rm dist}^2(\x,S_0), \\
        & \inprod{\u}{\bar\x} = 0, \, \inprod{\u}{\v} = 0,\, \mu\in [0,1].
        \end{aligned}
        \end{equation}
We proceed the proof by considering two cases where the dimension of problem \eqref{prob:zy1y0} is  $n = 2$ or $n\ge 3$.
\begin{enumerate}
\item[{\bf Case I:}] If $n = 2$, then for any feasible solution $(\mu,\lambda, \u)$ to \eqref{prob:mu}, it holds that $\u \equiv {\bf 0}$. Let $(\lambda^*,\mu^*,{\bf 0})$ be an optimal solution to \eqref{prob:mu}. Then,
$\z^*=\lambda^*\bar\x+\mu^*\v$ is an optimal solution to \eqref{prob:zy1y0}.
Define $\tilde \z = \lambda^*\beta\bar\x+\mu^*\v $,
where $\beta \ge 1$ is some constant such that $\norm{\tilde \z} = 1$.
Note that the above construction of $\tilde \z$ implies that $\Pi_{L}(\tilde \z) = \Pi_{L}(\z^*)$. Let $\tilde\theta\in(0,\pi/2)$ be the angle between $\tilde \z - \x_0$ and $\bar \x - \x_0$.
{Note that $\gamma\in(0,\pi/2)$ can be regarded as the angle between $\frac{\bar\x}{\|\bar\x\|}-\x_0$ and $\bar\x-\x_0$.}  Then, geometric arguments assert that $\tan \tilde \theta \ge \tan \gamma$, i.e,  $\tilde\theta \ge \gamma$ (see Figure \ref{fig:angles} for the illustration).
Hence,
\[
\norm{\tilde \z - \Pi_{L}(\tilde \z)} = \norm{\tilde \z -\Pi_{L}(\z^*)}
= \norm{\tilde \z - \x_0} \sin\tilde \theta \ge \norm{\tilde \z - \x_0}\sin \gamma.
\]
First consider the case where $\norm{\tilde \z - \x_0} \ge \norm{\z^* - \x_0}$.
In this case, we know that
\begin{equation}   \label{eq:distz}
\begin{aligned}
&   \sqrt{1 - \norm{\z^*}^2} + \norm{\z^* - \Pi_{L}(\z^*)} \ge \|\tilde \z-\z^*\|+ \norm{\z^* - \Pi_{L}(\z^*)}  \\
\ge&\norm{\tilde \z - \Pi_{L}(\z^*)} \ge \norm{\tilde \z - \x_0} \sin \gamma \ge \norm{\z^* - \x_0} \sin \gamma = {\rm dist}(\x, S_0) \sin\gamma,
\end{aligned}
\end{equation}
where the first inequality follows from the facts that $\inprod{\tilde \z-\z^*}{\z^*}=(\beta-1)(\lambda^*)^2\norm{\bar\x}^2\ge0$ and  $1=\|\tilde\z\|^2=\|\tilde \z-\z^*\|^{2}+\|\z^*\|^2+2\inprod{\tilde \z-\z^*}{\z^*}\ge\|\tilde \z-\z^*\|^{2}+\|\z^*\|^2$.
Next, we consider the case where $\norm{\tilde \z - \x_0} < \norm{\z^* - \x_0}$. To proceed, define by $\theta\in (0,\pi/2)$ the angle between $\z^* - \x_0$ and $\bar\x - \x_0$. %
{Since  $\tilde\theta\in (0,\pi/2)$, and  $\cos\theta=\frac{\|\x_0-\Pi_{L}(\z^*)\|}{\norm{\z^* - \x_0}} < \frac{\|\x_0-\Pi_{L}(\z^*)\|}{\norm{\tilde\z - \x_0}}= \cos\tilde\theta$, we see that $\sin\theta>\sin\tilde\theta$.}
Then, it holds that
\begin{equation}
\label{eq:distz2}
\norm{\z^* - \Pi_{L}(\z^*)} = \norm{\z^* - \x_0}\sin \theta > {\rm dist}(\x,S_0) \sin\tilde\theta  \ge {\rm dist}(\x, S_0)\sin \gamma.
\end{equation}
From \eqref{eq:distrelation}, \eqref{eq:distz} and \eqref{eq:distz2}, we see that \eqref{eq:distxx0xx1} holds.
\item[{\bf Case II:}] If $n\ge3$, we observe that $\u$ can be eliminated from problem \eqref{prob:mu}
by using the fact that $\|\bar\x\|^2+\|\v\|^2=\|\x_0\|^2=1$. In particular, problem \eqref{prob:mu} can be rewritten as follows:
\begin{equation}
\label{prob:mu_rd}
        \begin{aligned}
\min_{\mu,\lambda\in\R} \quad &\mu \\
{\rm s.t.}\quad & \lambda^2 \norm{\bar \x}^2 + \mu^2 \norm{\v}^2 \le  \norm{\x}^2, \\
& (\lambda -1)^2 \norm{\bar \x}^2 + (\mu -1)^2\norm{\v}^2 \le {\rm dist}^2(\x,S_0), \\
& \norm{\x}^2 +1 - 2\lambda\norm{\bar \x}^2 - 2\mu\norm{\v}^2 = {\rm dist}^2(\x, S_0),\\
& \mu\in [0,1],
\end{aligned}
\end{equation}
where the equality constraint comes from eliminating $\|\u\|^2$ in the first two constraints in \eqref{prob:mu}.
Indeed, for any feasible solution $(\lambda,\mu,\u)$ to \eqref{prob:mu}, $(\lambda,\mu)$ is feasible to \eqref{prob:mu_rd}. Meanwhile, if  $(\lambda,\mu)$ is feasible to \eqref{prob:mu_rd}, since $n\ge 3$, one can always find $\u\in\R^n$ satisfying $\|\u\|^2=\|\x\|^2-\big(\lambda^2 \norm{\bar \x}^2 + \mu^2 \norm{\v}^2\big) \ge 0$ and $\inprod{\u}{\bar\x} = 0, \, \inprod{\u}{\v} = 0$ such that
$(\lambda,\mu,\u)$ is a feasible solution to \eqref{prob:mu}.
Consequently, $(\lambda^*,\mu^*,\u^*)$ is an optimal solution to \eqref{prob:mu} if and only if $(\lambda^*,\mu^*)$ is an optimal solution to \eqref{prob:mu_rd}.

Let $\mu^*$ be the optimal value of problem \eqref{prob:mu_rd}. Then, $(\lambda^*,\mu^*)$ with $\lambda^* = \big(1 + \norm{\x}^2 - 2\mu^*\norm{\v}^2 - {\rm dist}^2(\x,S_0)\big)/2\norm{\bar\x}^2$ is the unique optimal solution to problem \eqref{prob:mu_rd}.
We consider three cases here:
\begin{itemize}
        \item[\rm (i)] $\mu^* \in (0,1)$.
        {Note that the first three constraints in \eqref{prob:mu_rd} result a line segment, denoted by  $F$, where its endpoints
        are the two intersection points (note that since $\x$ is  a feasible solution of \eqref{prob:zy1y0} with $\u={\bf 0}$, the two ellipses must intersect) of the  two ellipses:
\[
\left\{
(\lambda, \mu) \mid  \lambda^2 \norm{\bar \x}^2 + \mu^2 \norm{\v}^2 =  \norm{\x}^2
\right\}\]
and
\[
\left\{(\lambda, \mu) \mid  (\lambda -1)^2 \norm{\bar \x}^2 + (\mu -1)^2\norm{\v}^2 = {\rm dist}^2(\x,S_0)
\right\}.
\]
Hence the feasible set to \eqref{prob:mu_rd}  can be written as $ F\cap \{(\lambda,\mu) \mid 0\le\mu\le1\}$.
Since $\norm{\bar\x} \neq 0$, the equality constraint in \eqref{prob:mu_rd} implies that $F$ cannot be parallel to the $\lambda$-axis. Now, $\mu^* \in (0,1)$ implies that the optimal solution to \eqref{prob:mu_rd} must be an endpoint of $F$.
This further implies that the first two inequality constraints in \eqref{prob:mu_rd} are active at the optimal solution $(\lambda^*,\mu^*)$.}
Then, {$(\lambda^*,\mu^*, \bf 0)$} is an optimal solution to \eqref{prob:mu}.
Therefore, $\z^*=\lambda^*\bar\x+\mu^*\v$ is an optimal solution to problem \eqref{prob:zy1y0}. The desired result then follows from the same arguments in Case I.
\item [\rm (ii)] $\mu^* = 1$. In this case, we have from \eqref{eq:dist} that
$\norm{\z^* - \Pi_{L}(\z^*)} = {\rm dist}(\x, S_0).$
Hence, we have
\[
\sqrt{1 - \norm{\x}^2} + {\rm dist}(\x, S_1) \ge \sqrt{1 - \norm{\x}^2} + \norm{\z^* - \Pi_{L}(\z^*)} \ge {\rm dist}(\x, S_0) \ge {\rm dist}(\x, S_0)\sin\gamma,
\]
where the first inequality follows from the optimality of $\z^*$ to \eqref{prob:zy1y0} and ${\rm dist}(\x,S_1)=\|\x-\Pi_{S_1}(\x)\|=\|\x-\Pi_L(\x)\|$.
\item[\rm (iii)] $\mu^* = 0$. In this case, $(\lambda^*,0,\u^*)$ with some $\u^*$ satisfying $\|\u^*\|^2=\|\x\|^2- (\lambda^*)^2 \norm{\bar \x}^2 $ and $\inprod{\u^*}{\bar\x} = 0, \, \inprod{\u^*}{\v} = 0$ is an optimal solution to \eqref{prob:mu}.
Then, $\z^* = \lambda^*\bar\x + \u^*$ is an optimal solution to \eqref{prob:zy1y0} and
 $\Pi_{L}(\z^*) = \bar\x$. Let $\tilde \z := \z^* + \beta \v = \lambda^*\bar\x + \u^* + \beta\v$ with $\beta \ge 0$ such that
$\norm{\tilde \z} = 1$. Then, we see that
\begin{equation}
\label{eq:caseiii}
\norm{\tilde\z - \bar\x}^2 = (\lambda^* - 1)^2 \norm{\bar\x}^2 + \norm{\u^*}^2 + \beta^2 \norm{\v}^2
= 1 - 2\lambda^*\norm{\bar\x}^2 + \norm{\bar\x}^2 = {\rm dist}^2(\x,S_0),
\end{equation}
where the second equality holds since $\norm{\tilde\z} = 1$ and the third equality follows from the equality constraint in problem \eqref{prob:mu_rd}. Meanwhile, it holds that
\[
1 - \norm{\z^*}^2 = 1 - (\lambda^*)^2\norm{\bar\x}^2 - \norm{\u^*}^2  = \beta^2 \norm{\v}^2
=\norm{\tilde \z - \z^*}^2.
\]
Hence,  we have
\begin{align*}
\sqrt{1 - \norm{\x}^2} + {\rm dist}(\x, S_1) \ge{}& \sqrt{1 - \norm{\z^*}^2} + \norm{\z^* - \Pi_{L}(\z^*)} \\
={}&  \norm{\tilde \z - \z^*} + \norm{\z^* - \bar\x}
\\
\ge{}& \norm{\tilde \z - \bar\x} \\
\ge{}&{\rm dist}(\x, S_0) \sin \gamma,
\end{align*}
where the first inequality follows from {\eqref{eq:distrelation}}{, and the last inequality follows from \eqref{eq:caseiii}.}
\end{itemize}
\end{enumerate}
        We have shown that \eqref{eq:distxx0xx1} holds
        and thus completed the proof of Lemma \ref{lemma:xS}.
\end{proof}

{ 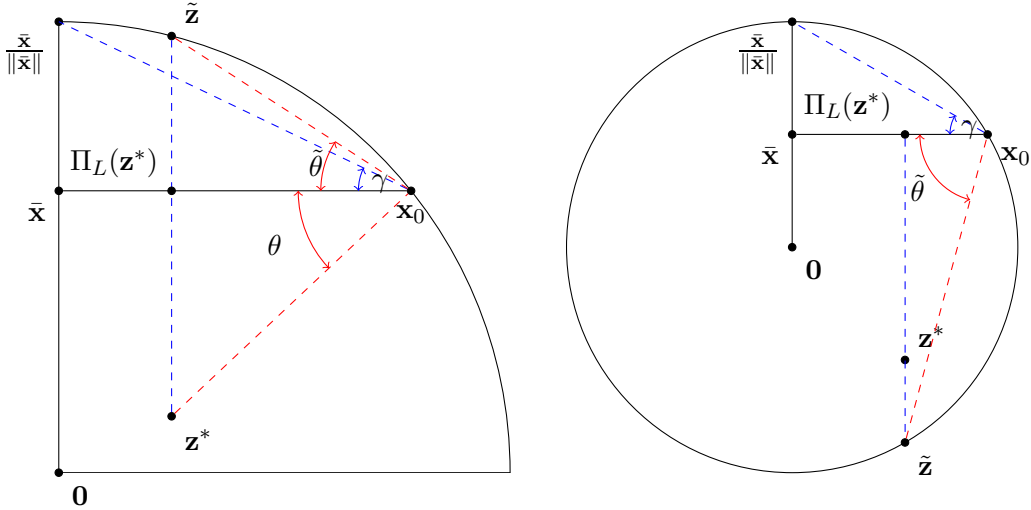
\begin{figure}
                \centering
                \begin{tikzpicture}%
                [
                scale=1.5,
                point/.style = {draw, circle,  fill = black, inner sep = 1pt},
                dot/.style   = {draw, circle,  fill = black, inner sep = .2pt},
                line width/.style = 1pt,
                ]

                \begin{scope}
                \clip (0,0) rectangle (4,4);
                \draw (0,0) circle(4);
                \draw (0,0) -- (4,0);
                \end{scope}

                \node (origin) at (0,0)[point, label = {below right:${\bf 0}$}]{};
                \draw (0,0) -- (0,4);
                \node (xbar) at (0,2.5) [point, label = {below left:$\bar\x$}]{};
                \node (norxbar) at (0,4)  [point, label = {below left:$\frac{\bar\x}{\|\bar\x\|}$}]{};
                \node (x0) at (3.1225,2.5)[point, label = {below:$\x_0$}]{};
                \draw (xbar) -- (x0);
        \draw[dashed, blue] (x0) -- (0,4);
        \node (zstar) at (1,0.5) [point, label = {below right:${\z^*}$}]{};
        \node (ztilde) at (1,3.8730) [point, label = {above right:${\tilde\z}$}]{};
         \draw[dashed, blue] (zstar) -- (ztilde);
         \draw[dashed, red] (ztilde) -- (x0);
         \draw[dashed, red] (zstar) -- (x0);
         \node (plzstar) at (1,2.5) [point, label = {above left:${\Pi_{L}(\z^*)}$}]{};
         \node (xbarend) at (0,4)[]{};
         \node (xbarstart) at (0,1)[]{};
         \draw pic["$\gamma$", draw=blue!100, <->, angle eccentricity=0.6, angle radius=0.7cm]
         {angle=xbarend--x0--xbar};

         \draw pic["$\tilde \theta$", draw=red!100, <->, angle eccentricity=1.1, angle radius=1.2cm]
         {angle=ztilde--x0--xbar};

         \draw pic["$\theta$", draw=red!100, <->, angle eccentricity=1.3, angle radius=1.5cm]
         {angle=xbar--x0--zstar};

         \draw (6.5,2) circle (2);
         \node (origin2) at (6.5,2)[point, label = {below right:${\bf 0}$}]{};
         \draw (6.5,2) -- (6.5,4);
         \node (xbar1) at (6.5,3) [point, label = {below left:$\bar\x$}]{};
         \node (norxbar1) at (6.5,4)  [point, label = {below left:$\frac{\bar\x}{\|\bar\x\|}$}]{};
         \node (x01) at (8.2321,3)[point, label = {below right:$\x_0$}]{};
         \draw (xbar1) -- (x01);
         \draw[dashed, blue] (x01) -- (6.5,4);
         \node (zstar1) at (7.5,1) [point, label = {above right:${\z^*}$}]{};
         \node (ztilde1) at (7.5,0.2679) [point, label = {below right:${\tilde\z}$}]{};
                 \node (xbarend1) at (6.5,4)[]{};
         \draw[dashed, red] (ztilde1) -- (x01);
         \node (plzstar1) at (7.5,3) [point, label = {above left:${\Pi_{L}(\z^*)}$}]{};
          \draw[dashed, blue] (plzstar1) -- (ztilde1);
          \draw pic["$\gamma$", draw=blue!100, <->, angle eccentricity=0.5, angle radius=0.5cm]
          {angle=xbarend1--x01--xbar1};

          \draw pic["$\tilde \theta$", draw=red!100, <->, angle eccentricity=1.3, angle radius=0.9cm]
          {angle=xbar1--x01--ztilde1};
                \end{tikzpicture}
                \caption{Illustration of two possible scenarios of the positions of $\tilde\z$. In either case, it holds that $\tilde\theta \ge \gamma$.
                {In the first case, it also holds that $\cos\theta\le\cos\tilde \theta$ (or equivalently, $\sin\theta\ge\sin\tilde\theta$).}}
                \label{fig:angles}
\end{figure}}

Note that Lemma \ref{lem:EBnonneg} provides certain error bound inequality involving ${\rm dist}(\x,S_1)$ for the convex problem ($\rm P_1$) and Lemma \ref{lemma:xS} connects ${\rm dist}(\x,S_1)$ and ${\rm dist}(\x,S_0)$. Using these results, we obtain in the following lemma the desired error bound result in the hard case 2 (ii).
\begin{lem}
\label{lem:h2ii}
In the hard case 2 (ii), there exists a constant $\tau>0$ such that for all $\x\in \R^n$,
\begin{equation*}
\dt(\x,S_0)\le \tau\big(f(\x)-f^*+\delta_{B}(\x)\big)^{1/2}.
\end{equation*}
\end{lem}
\begin{proof}
Recall that $\tilde f(\x):=f(\x)-\lambda_1(\x^T\x-1)$ is convex. Note that in hard case 2 (ii), it holds that $\min \tilde f(\x) = f^*$ and $\|\bar\x\|<1$. Then, by Lemma \ref{lem:EBnonneg} item 3, we know that there exist a constant $\tau_1 > 0$ such that
\begin{equation}
\label{eq:xS1}
\dt(\x,S_1)\le \tau_1\big(\tilde f(\x)-f^*+\delta_B(\x)\big)^{1/2}, \quad \forall\, \x\in\R^n.
\end{equation}
Without loss of generality, one can assume that $\tau_1 \ge 1/\sqrt{-\lambda_1}$, i.e., $2\tau_1^2 \lambda_1 + 2 \le 0$.
By Lemma \ref{lemma:xS}, we know that
\[
{\rm dist}(\x,S_0)\sin\gamma \le {\rm dist}(\x,S_1) + \sqrt{1 - \norm{\x}^2}, \quad \forall \norm{\x} \le 1,
\]
where $\sin\gamma = \frac{1-\|\bar\x\|}{ \sqrt{(1 - \norm{\bar\x})^2 + 1-\|\bar\x\|^2}}$.
Together with \eqref{eq:xS1}, this implies that for all $\norm{\x}\le 1$,
\begin{align*}
{\rm dist}^2(\x,S_0)\sin^2\gamma \le {}& 2 {\rm dist}^2(\x,S_1) + 2(1 -\norm{\x}^2) \\
\le {}& 2\tau_1^2(\tilde f(\x) - f^* + \delta_{B}(\x)) + 2(1 - \norm{\x}^2)\\
= {}& 2\tau_1^2(f(\x) - f^*+\delta_{B}(\x)) + (2 + 2\tau_1^2\lambda_1)(1 - \norm{\x}^2) \\
\le {}& 2\tau_1^2(f(\x) - f^*+\delta_{B}(\x)),
\end{align*}
where the last inequality holds since $2 + 2\tau_1^2\lambda_1\le 0$.
Thus, we know that
\[
{\rm dist}(\x,S_0) \le \frac{\sqrt{2}\tau_1}{\sin \gamma}\big(f(\x) - f^* + \delta_B(\x)\big)^{1/2}, \quad \forall\, \x\in \R^n.
\]
We complete the proof of this lemma.
\end{proof}

Now, with Lemmas \ref{lem:EBnonneg}, \ref{lem:ezh1}, \ref{lem:h2i} and \ref{lem:h2ii}, we are able to summarize the situations in the following \emph{TRS-ill case} in which the H\"olderian error bound modulus is 1/4:
\begin{equation}
\label{eq:con}
\lambda_1\le 0,~\b\in\ra(\A-\lambda_1\I) \text{ and } \norm{(\A-\lambda_1\I)^\dagger \b} = 1.
\end{equation}
After all these preparations, we arrive at the following theorem.
\begin{thm}
\label{thm:main}
For the trust region subproblem problem $\rm(P_0)$, there exists a constant $\tau_{EB} >0$ such that
\[\dt(\x,S_0)\le\tau_{EB} \big(f(\x)-f^*+\delta_{B}(\x)\big)_+^\rho,~\forall \, \x\in\R^n,\]
where %
$\rho=\left\{
\begin{aligned}
& 1/4, \mbox{ for the  TRS-ill case  \eqref{eq:con}}, \\[5pt]
& 1/2, \mbox{ otherwise.}
\end{aligned}
\right.$
\end{thm}
The above theory in fact shows that  the H\"olderian error bound always holds with modulus 1/4 for all cases due to the facts that the function value $f(\x)$ is bounded in the unit ball and the inequality  $t^{1/2}\le t^{1/4}$ holds for all $t\in(0,1)$.
\section{KL inequality for the TRS}
In this section, based on the previously developed error bound results, we derive the KL inequality associated with the TRS.
To prove the KL inequality, let us first recall part of the results from \cite{bolte2017error} that essentially state that in the \emph{convex} setting,  H\"olderian error bounds imply the KL inequality.
In fact, the equivalence between these two concepts for \emph{convex} problems is also obtained in \cite{bolte2017error}.
\begin{lem}[Corollary 6 (ii) in \cite{bolte2017error}]
\label{lem:eb2kl}
Let H be a real Hilbert space. Set
\[\mathcal{K}(0,+\infty)=\{\varphi\in C^0[0,+\infty)\cap C^1(0,+\infty),\varphi(0)=0,\varphi\text{ is concave and }\varphi'>0\}.\]
 Let $h:H\rightarrow(-\infty, +\infty]$ be a proper, convex and lower-semicontinuous function, with $\min h=0$. Let $\varphi\in \mathcal{K}(0,+\infty),c>0$. Let $S=\{x\in H\mid h(x)=\min h(x)\}$.

Then if $s\varphi'(s)\ge c\varphi(s)$ for all $s>0$, and $\varphi(h(x))\ge\dt(x,S)$ for all $x\in \left\{x\in H\mid h(x) >0 \right\}$, then $\varphi'(h(x)){\rm dist}(0,\partial h(x))\ge c$ for all $x\in \left\{x\in H\mid h(x) >0 \right\}$.
\end{lem}
{Recall that we denote by $f^*$ the optimal value of the TRS $\rm(P_0)$.} With Lemma \ref{lem:eb2kl}, we are ready to prove the KL inequality for the convex TRS. As one can observe from Theorem \ref{thm:kl-convex}, the KL inequality for the convex TRS holds globally.
\begin{thm}
        \label{thm:kl-convex}
        For the TRS $({\rm P}_0)$ with $\lambda_1 \ge 0$, there exists some constant $\tau >0$ such that the KL inequality holds
        \begin{equation}
        \label{eq:KL-convex}
        \big(f(\x)-f^*+\delta_{B}(\x)\big)^{1-\rho}\le\tau {\rm dist}(-\nabla f(\x), N_B(\x)), \quad \forall\, \x \in \R^n,
        \end{equation}
where $\rho = \frac{1}{4}$ if $\lambda_1=0$, $\b\in{\rm Range}(\A)$ and $\norm{\A^{\dagger}\b} = 1$; otherwise $\rho = \frac{1}{2}$.
\end{thm}
\begin{proof}
        By Theorem \ref{thm:main}, we know that there exists a constant $\tau >0$ such that for all $\x \in \R^n$
        \begin{equation}\label{eq:ebcvx}
        \dt(\x,S_0)\le \tau \big(f(\x)-f^*+\delta_{B}(\x)\big)^{\rho}
        \end{equation}
        with $\rho = \frac{1}{4}$ if  $\lambda_1=0$, $\b\in{\rm Range}(\A)$ and $\norm{\A^{\dagger}\b} = 1$; and  $\rho = \frac{1}{2}$ otherwise.
        Now define $\varphi(s) = \tau s^\rho$ for all $s\ge 0$. Obviously, $\varphi\in {\cal K}(0,+\infty)$ and we also note that
        \[
        s\varphi'(s) = \tau s \rho s^{\rho - 1} = \tau \rho s^{\rho} = \rho \varphi(s), \quad \forall\, s> 0.
        \]
        Hence, by \eqref{eq:ebcvx}, Lemma \ref{lem:eb2kl} and the fact that $\dt(0,\partial [f(\x)+\delta_B(\x)])=\dt(-\nabla f(\x),N_B(\x))$, we have that for all $\x\not\in S_0$,
        \[
        \tau \rho (f(\x)-f^*+\delta_{B}(\x))^{\rho-1} {\rm dist}(-\nabla f(\x), N_B(\x)) \ge \rho.
        \]
        Thus, by noting that \eqref{eq:KL-convex} holds trivially if $\x\in S_0$, we complete the proof for the theorem.
\end{proof}

{Next, we consider the nonconvex case.
In this case, the analysis is more complicated.
We note that in \cite{drusvyatskiy2014second}, the authors show that under the prox-regularity \cite{poliquin2014generalized} assumption, for extended-real-valued lower semicontinuous functions, the error bound condition with modulus $1/2$ (coined as the quadratic growth condition in \cite{drusvyatskiy2014second}), implies the metric subregularity \cite{dontchev2009implicit} property. However, for the TRS, the relations between the metric subregularity property  and the KL inequality remain largely unknown. Hence, the results in \cite{drusvyatskiy2014second} can not be directly used here.

In the following, we give a comprehensive analysis of the KL inequality for the nonconvex TRS.}
For ease of reading, we first provide here a roadmap of the analysis. Apart from some general discussions and computations, our proof can be divided into mainly three steps. Specifically,
in the first step, we establish the KL inequality for testing points chosen from the intersection of the interior of the unit norm ball and a neighborhood of the given optimal reference point.
Then, we restrict our attentions on the testing points chosen from the boundary of the unit norm ball. We discuss the easy case and hard case 1 in the second step. The analysis for the hard case 2 is presented in the last step where certain parts of the proofs are presented in the Appendix.

We begin our analysis with some general discussions.
Let $\x^*$ be an optimal solution to the TRS $({\rm P}_0)$ with $\lambda_1 < 0$. We know from the KKT condition in Section \ref{sec:pre} that there exists $\lambda^* \ge 0$ such that
\begin{equation}
\label{eq:optnonconvex}
\nabla f(\x^*) + 2\lambda^* \x^* = 2\A\x^*-2\b+ 2\lambda^*\x^*= {\bf 0}, \,  \lambda^* \ge -\lambda_1 > 0, \, \lambda^*(1 - \norm{\x^*}) = 0.
\end{equation}
Thus, $\norm{\x^*} = 1$ and $\norm{\nabla f(\x^*)} = 2\lambda^*$.
For any given $\x\in\R^n$, since
$$\|\nabla f(\x)-\nabla f(\x^{*})\|=\| 2\A(\x-\x^*)\|\le 2\|\A\|_2\|\x-\x^*\|,$$ it holds that
\begin{equation}
\label{eq:gradf}
\|\nabla f(\x)\|\ge\|\nabla f(\x^{*})\|-2\|\A\|_2\|\x-\x^*\|.
\end{equation}
Meanwhile since $f(\x)$ is a quadratic function, we have
\begin{equation}
\label{eq:valf}
f(\x)-f(\x^{*})= \inprod{\nabla f(\x^*)}{\x-\x^*}+ \inprod{\x - \x^*}{\A(\x - \x^*)}.
\end{equation}
	Now, for any $\x$ satisfying $\|\x-\x^{*}\|\le\frac{\lambda^*}{\|\A\|_2+\lambda^*}$ and $\norm{\x}\le 1$, we know from \eqref{eq:gradf} and \eqref{eq:valf} that
\begin{equation*}
\left\{
\begin{aligned}
&\norm{\nabla f(\x)} \ge 2\lambda^*  - \frac{2\lambda^*\norm{\A}_2}{\|\A\|_2+\lambda^*} = \frac{2(\lambda^*)^2}{\|\A\|_2+\lambda^*} = \frac{\lambda^*}{\|\A\|_2+\lambda^*} \norm{\nabla f(\x^*)}, \\
& f(\x)-f(\x^{*})\le \frac{2(\lambda^*)^2}{\|\A\|_2+\lambda^*}
+ \frac{\norm{\A}_2 (\lambda^*)^2}{(\|\A\|_2+\lambda^*)^2}
\le 2\lambda^*  = \norm{\nabla f(\x^*)},
\end{aligned}
\right.
\end{equation*}
and consequently,
\begin{equation}
\label{eq:finteiror}
\big(f(\x) - f(\x^*)\big)^{\frac{1}{2}} \le \|\nabla f(\x^*)\|^{1/2} \le  \frac{\|\A\|_2+\lambda^*}{\sqrt2(\lambda^*)^{3/2}}\|\nabla f(\x)\|.
\end{equation}
{Then we have the following results associated with the case where the testing points are chosen from the intersection of the interior of the unit norm ball and a neighborhood of the given optimal reference point $\x^*.$}
\begin{lem}
\label{lem:interiorKL}
{Suppose  $\lambda_1 <0$}. Then there exist some constant $\tau >0$ and sufficiently small $\epsilon >0$ such that
	\begin{equation*}
	\label{eq:interiorKL}
	\big(f(\x)-f^*+\delta_{B}(\x)\big)^{1/2}\le\tau{\rm dist}(-\nabla f(\x), N_B(\x)), \quad \forall\, \x\in B(\x^*,\epsilon)\cap \left\{\x\in\R^n \mid \norm{\x} < 1 \right\}.
	\end{equation*}
\end{lem}
\begin{proof}
Note that for all $\x$ satisfying $\norm{\x} < 1$, it holds that $N_B(\x) = \{{\bf 0}\}$ and ${\rm dist}(-\nabla f(\x),N_B(\x)) = \norm{\nabla f(\x)}$. Hence, we know from \eqref{eq:finteiror} that
	\begin{equation*}
	\label{eq:KLnormxs1}
	\big(f(\x)-f^*+\delta_{B}(\x)\big)^{1/2} \le \frac{\norm{\A}_2 + \lambda^*}{\sqrt2(\lambda^*)^{3/2}} {\rm dist}(-\nabla f(\x), N_B(\x)),
	\end{equation*}
	whenever
	 $\|\x-\x^{*}\|\le\frac{\lambda^*}{\|\A\|_2+\lambda^*}$ and $\norm{\x}< 1$.
\end{proof}

In the subsequent discussions, we will focus on the boundary of the unit norm ball, denoted by ${\bf bd}(B):=\left\{\x\in\R^n\mid \norm{\x} = 1 \right\}$. For all $\x\in {\bf bd}(B)$, it {follows from \eqref{eq:valf} } that
\begin{equation}\label{eq:fmfstar}
f(\x) - f(\x^*) = -\inprod{2\lambda^* \x^*}{\x - \x^*} +  \inprod{\x - \x^*}{\A(\x - \x^*)} = \inprod{\x - \x^*}{(\A + \lambda^*\I) (\x - \x^*)},
\end{equation}
{where the second equality holds since $\|\x\|=\|\x^*\|=1.$}
Meanwhile, for any given $\x\in {\bf bd}(B)$, we have $N_B(\x) = \{\nu\x \in \R^n \mid \nu \ge 0\}$
and thus
\begin{equation}
\label{eq:distnorm1}
{\rm dist}(-\nabla f(\x), N_B(\x)) = \min_{\nu \ge 0} \norm{ 2(\A\x - \b) + \nu\x}.
\end{equation}
Let $\nu^*(\x)$ be the optimal solution to problem \eqref{eq:distnorm1}. It can be verified that $\nu(\x) = \max\{\inprod{-2(\A\x - \b)}{\x},0\}$. Moreover, it holds that
\begin{equation}
\label{eq:nux}
\nu(\x) \to \nu(\x^*) =
\inprod{-2(\A\x^* -\b)}{\x^*} = 2\lambda^* >0 \, \mbox{ as } {\bf bd}(B) \ni \x\to \x^*.
\end{equation}
Hence, we know that there exists a constant $\epsilon_0 \in (0,1)$ such that
\begin{equation}
\label{eq:eps0}
\nu(\x) = \inprod{-2(\A\x - \b)}{\x} >0, \quad \forall\, \x \in {\bf bd}(B) \cap \left\{\x\in\R^n\mid \norm{\x -\x^*}\le \epsilon_0 \right\}.
\end{equation}
Now, we are ready to prove the desired KL inequality for the easy case and hard case 1.
\begin{lem}[KL inequality for the easy case and hard case 1]
	\label{lem:KLeasyhard1}
	Suppose $\lambda_{1} <0$. In the easy case and hard case 1, there exist some constant $\tau >0$ and sufficiently small $\epsilon >0$ such that
	\begin{equation}
	\label{eq:easkyhard1}
	\big(f(\x)-f^*+\delta_{B}(\x)\big)^{1/2}\le\tau{\rm dist}(-\nabla f(\x), N_B(\x)), \quad \forall\, \x\in B(\x^*,\epsilon).
	\end{equation}
\end{lem}
\begin{proof}
Note that in the easy case and hard case 1, it holds that
$\lambda^* > -\lambda_1 >0$ and $\norm{(\A - \lambda_1\I)^{\dagger}\b}\neq 1$.
 Since $\lambda^* > -\lambda_1 >0$, from the discussion of \eqref{eq:nux}, we know that there exists a positive constant $\epsilon_1 \le \epsilon_0$ such that
	$\nu(\x) + 2\lambda_1 > 0$ whenever $\norm{\x} = 1$ and $\norm{\x - \x^*}\le \epsilon_1$. Recall the objective function $\tilde f$ in the convex problem (${\rm P}_1$), i.e., $\tilde f(\x):= \x^T(\A-\lambda_1\I)\x-2\b^T\x+\lambda_1$.
	We note for all $\x$ satisfying $\norm{\x} = 1$ and $\norm{\x - \x^*}\le \epsilon_1$ that
	\begin{align*}
	{\rm dist}(-\nabla \tilde f(\x), N_B(\x)) ={}& \min_{\mu\ge 0} \|2(\A-\lambda_{1}\I)\x-2\b+\mu\x\| \\
	={}& \|2(\A-\lambda_{1}\I)\x-2\b+(-\inprod{2 (\A-\lambda_1 \I)\x - 2 \b}{\x})\x\| \\
	={}& \|2(\A-\lambda_{1}\I)\x-2\b+(\nu(\x) + 2\lambda_1)\x\| \\
	={}& \norm{2(\A\x - \b) + \nu(\x)\x} \\
    ={}&{\rm dist}(-\nabla f(\x), N_B(\x)),
	\end{align*}
{where the second equality is due to $-\inprod{2 (\A-\lambda_1 \I)\x - 2 \b}{\x}=\nu(\x)+2\lambda_1>0.$}
	Note that $\tilde f(\x) = f(\x), \, \forall\, \x \in {\bf bd}(B)$ and from Lemma \ref{lemma:P1} that $f^*$, the optimal value of the TRS (${\rm P}_0$), is also the optimal value of the convex problem (${\rm P}_1$).
	Since $\norm{(\A - \lambda_1\I)^{\dagger}\b}\neq 1$, we obtain by Theorem \ref{thm:kl-convex} that there exists a positive constant $\tau_1$ such that for all $\x\in{\bf bd}(B)$ with $\norm{\x - \x^*}\le \epsilon_1$,
	\begin{equation*}
	\label{eq:easyKL}
	\begin{aligned}
	\big(f(\x)-f^*+\delta_{B}(\x)\big)^{\frac{1}{2}} ={}& \big(\tilde f(\x)-f^*+\delta_{B}(\x)\big)^{\frac{1}{2}} \\[2pt]
	\le{}&\tau_1 {\rm dist}(-\nabla \tilde f(\x), N_B(\x)) = \tau_1{\rm dist}(-\nabla f(\x), N_B(\x)).
	\end{aligned}
	\end{equation*}
{This, together with Lemma \ref{lem:interiorKL}, implies that \eqref{eq:easkyhard1} holds with some positive constant $\epsilon$ and $\tau$.}
\end{proof}

Now, we are ready for presenting the analysis for the hard case 2. {Recall $\lambda^*=-\lambda_1$ in the hard case.} First, we discuss the case where $\A - \lambda_1 \I ={\bf  0}$.
In this case, we know from the optimality condition \eqref{eq:optnonconvex} that $\b= {\bf 0}$ and thus $f(\x^*) = f(\x) =\lambda_1 \x^T\x = \lambda_1$ whenever $\norm{\x} = 1$. This, together with Lemma \ref{lem:interiorKL}, implies that the KL inequality holds with the exponent $1/2$.
Note that this case belongs to the hard case 2 (ii) as $\b\perp\Null(\A-\lambda_1\I)$ and $\|(\A-\lambda_1\I)^\dagger\b\|=0<1$.

Hence, in the subsequent analysis, we focus on the nontrivial case where $\A - \lambda_1\I \neq {\bf 0}$.
Recall the spectral decomposition of $\A$ as $\A = \P{\mathbf \Lambda} \P^T$ in Section \ref{sec:pre}. Then, $\A - \lambda_1 \I = \P ({\mathbf \Lambda} - \lambda_1\I) \P^T$ is the spectral decomposition of $\A -\lambda_1\I$. Denote the diagonal matrix $\D = {\mathbf \Lambda} - \lambda_1\I = {\rm Diag}(\d)$.
Since $\A - \lambda_1 \I \neq 0$,  there exists some integer $1 \le K <n$ such that
\begin{equation}
\label{eq:kD}
\left\{\begin{array}{llll}
\d_i &=&0 &\mbox{ for } i=1,\ldots,K,\\
\d_i &=& \lambda_i -\lambda_1 &\mbox{ for } i=K+1,\ldots, n,
\end{array}\right.
\end{equation}
and $\d_n \ge \cdots \ge \d_{K+1} > 0$.
{
For all $\x\in {\bf bd}(B)$, we define
\begin{equation}\label{eq:defalbet}
\alpha(\x) :=(\x^*)^T(\A-\lambda_1\I)(\x-\x^*) \mbox{ and }
\beta(\x) :=(\x-\x^*)^T(\A-\lambda_1\I)(\x-\x^*).
\end{equation}
For simplicity in the following discussion, we sometimes suppress the dependence on $\x$ in our notation. For example, we often write  $\alpha(\x)$ and $\beta(\x)$ as $\alpha$ and $\beta$, respectively. However, this should not cause any confusion because the reference vector $\x$ will always be clear from the context.}
Let $\s=\P^T\x^*$ and $\z=\P^T(\x-\x^*)$ and we verify from \eqref{eq:defalbet} that
\begin{equation}\label{eq:defabP}
\left\{
\begin{array}{lllll}
\alpha &=&(\x^*)^{T}(\A-\lambda_1\I)(\x-\x^*) &=& \inprod{\s}{\D \z},\\
\beta  &=&(\x-\x^*)^T(\A-\lambda_1\I)(\x-\x^*)\overset{\eqref{eq:fmfstar}}=f(\x)-f^* &=&\inprod{\z}{\D\z}. \\
\end{array}
\right.
\end{equation}
Recall the definition of $\epsilon_0$ in \eqref{eq:eps0}. It holds that for any $\x\in {\bf bd}(B)$ with $\norm{\x - \x^*}\le \epsilon_0$,
\begin{equation}
\label{eq:klpf1}
\begin{array}{ll}
&{\rm dist}^2(-\nabla f(\x), N_B(\x))\\
=&\|2(\A\x-\b)-2\lambda_1\x+(2\lambda_1+\nu(\x))\x\|^2\\
=&\|2(\A-\lambda_1\I)(\x-\x^*)+2\left[(\x^*)^T(\A\x^*-\b)-\x^T(\A\x-\b)\right]\x\|^2\\
=&\|2(\A-\lambda_1\I)(\x-\x^*)+2\left[f(\x^*)-f(\x)+(\x^*-\x)^T(\A-\lambda_1\I)\x^*\right]\x\|^2\\
=&\|2(\A-\lambda_1\I)(\x-\x^*)-2(\beta+\alpha)\x\|^2\\
=&\|2(\A-\lambda_1\I)(\x-\x^*)\|^2-8\x^T(\A-\lambda_1\I)(\x-\x^*)(\beta+\alpha)+\norm{2\beta\x+2\alpha\x}^2\\
=&\|2(\A-\lambda_1\I)(\x-\x^*)\|^2-4\left(\beta+\alpha\right)^2\\
=&4(\norm{\D\z}^2 - (\alpha + \beta)^2),%
\end{array}
\end{equation}
where the second equality follows from the fact that $2\A\x^*-2\b-2\lambda_1\x^*={\bf 0}$ {(which is \eqref{eq:optnonconvex} with $\lambda^*=-\lambda_1$)} and the definition of $\nu(\x)$,
the third equality follows from $2\A\x^*-2\b-2\lambda_1\x^*={\bf 0}$ and \eqref{eq:fmfstar},
the forth equality holds due to  \eqref{eq:defalbet}, the fifth equality holds since $\x^T(\A-\lambda_1\I)(\x-\x^*)=\alpha+\beta$ and $\norm{\x} = 1$, and the last equality holds as $\z = \P^T(\x - \x^*)$.

Define function $H:{\bf bd}(B) \to \R$ as:
\begin{equation}\label{eq:defH}
H(\x):=\|(\A\x-\b)-\lambda_1\x+(\lambda_1+\nu(\x)/2)\x\|^2\overset{\eqref{eq:klpf1}}=\|\D\z-(\alpha+\beta)\x\|^2=\|\D\z\|^2-(\alpha+\beta)^2,
\end{equation}
and thus we have
\begin{equation}\label{eq:distandH}
{\rm dist}^2(-\nabla f(\x), N_B(\x))=4H(\x), \quad \forall\, \x \in {\bf bd}(B)\cap \left\{\x\in\R^n\mid \norm{\x - \x^*}\le \epsilon_0 \right\}.
\end{equation}

In the next lemma, we show that the KL inequality holds with an exponent $1/2$ for the hard case 2 (ii).
\begin{lem}[KL inequality for the hard case 2 (ii)]
	\label{lem:hard2ii}
	Suppose that $\lambda_1 <0$. In the hard case 2 (ii), there exist some constant $\tau >0$ and sufficiently small $\epsilon >0$ such that
	\begin{equation*}
	\label{eq:hardKL}
	\big(f(\x)-f^*+\delta_{B}(\x)\big)^{1/2}\le\tau{\rm dist}(-\nabla f(\x), N_B(\x)), \quad \forall\, \x\in B(\x^*,\epsilon).
	\end{equation*}
\end{lem}
\begin{proof}
We have discussed that the results hold under the case $\A - \lambda_1\I ={\bf 0}$. In the remaining, we consider the case where $\A - \lambda_1\I \neq {\bf 0}$.
In the hard case 2 (ii), since $\x^* \not\perp {\rm Null}(\A - \lambda_1 \I)$, i.e., $\s \not\perp {\rm Null}(\D)$, we know that
$
\sum_{j=1}^{K} \s_j^2>0,
$
and
 from Lemma \ref{lem:funcPsi} in Appendix that there exists $0 < \epsilon' < \epsilon_0$ such that for any $\x \in {\bf bd}(B)$ with $\norm{\x - \x^*}\le \epsilon'$,
\begin{equation*}\label{eq:Hhc2ii}
H(\x) \ge \frac{1}{2} \d_{K+1}(\sum_{j=1}^{K} \s_j^2) \beta,
\end{equation*}
i.e.,
\[
{\rm dist}^2(-\nabla f(\x), N_B(\x)) = 4H(\x) \ge 2\d_{K+1} (\sum_{j=1}^{K} \s_j^2) (f(\x) -f^*).
\]
This, together with Lemma \ref{lem:interiorKL}, competes the proof.
\end{proof}

Before discussing the general results for the hard case 2 (i), we state in the following lemma that in the hard case 2 (i) if the test point $\x$ is restricted in the subspace ${\rm Range}(\A - \lambda_1\I)$, the desired KL inequality can be derived by combining Lemma \ref{lem:KLeasyhard1} with a reduction strategy. To avoid the tedious reduction arguments, we present the proof for the lemma in the Appendix.

\begin{lem}
	\label{lem:KLld}
	Suppose $\lambda_1 <0$. In the hard case 2 (i), there exist some constant $\tau >0$ and sufficiently small $\epsilon >0$ such that
	\begin{equation*}
	\label{eq:easkyhard2ild}
	\big(f(\x)-f^*+\delta_{B}(\x)\big)^{1/2}\le\tau {\rm dist}(-\nabla f(\x), N_B(\x)), \quad \forall\, \x\in B(\x^*,\epsilon)\cap {\rm Range}(\A - \lambda_1 \I).
	\end{equation*}
\end{lem}

After these preparations, we are ready to show in the next lemma that the KL inequality holds with the exponent $3/4$ for the hard case 2 (i), which is a subcase of the TRS-ill case \eqref{eq:con}.
\begin{lem}[KL inequality for the hard case 2 (i)]
	\label{lem:hard2i}
	Suppose that $\lambda_1 <0$. In the hard case 2 (i), there exist some constant $\tau >0$ and sufficiently small $\epsilon >0$ such that
	\begin{equation}
	\label{eq:hardKL2i}
	\big(f(\x)-f^*+\delta_{B}(\x)\big)^{3/4}\le\tau{\rm dist}(-\nabla f(\x), N_B(\x)), \quad \forall\, \x\in B(\x^*,\epsilon).
	\end{equation}
\end{lem}
\begin{proof}
Note that in the hard case 2 (i), it holds that
$\x^*\perp {\rm Null}(\A - \lambda_1 \I)$ (or equivalently, $\x^*\in {\rm Range}(\A - \lambda_1 \I)$ or $\s\in\ra(\D)$).  Recall the definition of $\epsilon_0$ in \eqref{eq:eps0} and denote $\epsilon_1 : = \min \left\{
\epsilon_0, 1/2 \right\}$.
For any $\x\in{\bf bd}(B)$ with
$\norm{\x - \x^*}\le \epsilon_1$, we know from $\z = \P^T(\x-\x^*)$ that $\gamma:=\|\z\|^2 \le \epsilon_1^2 \le 1/4$.
Recall the definitions of $\alpha$ and $\beta$ in \eqref{eq:defalbet}.
In the following, we consider the following two cases, i.e., {\bf Case I:} with $|\alpha+\beta|\le\frac{1}{2}\sqrt{\d_{K+1}}\beta^{1/2}$ and {\bf Case II:} with $|\alpha+\beta|>\frac{1}{2}\sqrt{\d_{K+1}}\beta^{1/2}$:

{\bf Case I.} If $|\alpha+\beta|\le\frac{1}{2}\sqrt{\d_{K+1}}\beta^{1/2}$,  from \eqref{eq:kD}, \eqref{eq:defabP} and \eqref{eq:defH}, we have
\begin{equation}\label{eq:hc2ii1}
H(\x)=\|\D \z\|^{2}-(\alpha+\beta)^2
\ge\d_{K+1}\beta-\frac{1}{4}\d_{K+1}\beta = \frac{3}{4}\d_{K+1}\beta, \quad \forall\, \x \in {\bf bd}(B) \mbox{ with } \norm{\x - \x^*}\le \epsilon_1.
\end{equation}

{\bf Case II.} Now suppose  $|\alpha+\beta|>\frac{1}{2}\sqrt{\d_{K+1}}\beta^{1/2}$. Note that $\alpha=\inprod{\s}{\D\z}\to0$ and $\beta=\inprod{\z}{\D\z}\to0$ as $\z\to {\bf 0}$ or equivalently $\x \to \x^*$.
Then, by reducing $\epsilon_0$ if necessary, we have
 $|\alpha|\ge\beta$ and thus
 \begin{equation}
 \label{eq:algebe}
 |\alpha|\ge\frac{1}{4}\sqrt{\d_{K+1}}\beta^{1/2} \quad ~~~~\forall~ \x\in{\bf bd}(B) \mbox{ with }\norm{\x - \x^*}\le \epsilon_1.
 \end{equation}
Then if $\alpha\le0$, we know that for any $\x\in{\bf bd}(B)$ with
$\norm{\x - \x^*}\le \epsilon_1$,
\begin{equation}\label{eq:alphage0}
H(\x)=\|\D\z\|^2-\alpha^2-2\alpha\beta-\beta^2\ge 0-2\alpha\beta+\alpha\beta=-\alpha\beta\ge \frac{1}{4}\sqrt{\d_{K+1}}\beta^{3/2},
\end{equation}
where the first inequality holds as $\|\D\z\|^2-\alpha^2\ge(\s^T \D\z)^2-\alpha^2=0$ and $|\alpha|\ge\beta$.

Next consider the case with $\alpha>0$.
Recall that in the hard case 2 (i), we have $\s\in\ra(\D)$.
Hence, by reducing $\epsilon_0$ if necessary, we know from Lemma \ref{lem:KLld}, \eqref{eq:defabP}, \eqref{eq:defH} and \eqref{eq:distandH} that there exists some constant $c_1 \in (0,1)$ such that for any $\x \in{\bf bd}(B)\cap {\rm Range}(\A - \lambda_1 \I)$ with $\norm{\x - \x^*}\le \epsilon_1$,
\begin{equation}
\label{eq:KLrange}
\frac{1}{4}{\rm dist}^2(-\nabla f(\x), N_B(\x))= H(\x) = \norm{\D \z - (\alpha + \beta)(\s+\z)}^2 \ge c_1^2 \d_{K+1}\beta.
\end{equation}
For any given $\x \in {\bf bd}(B)$ satisfying $\norm{\x - \x^*}\le \epsilon_1$ {and $\x\notin{\rm Range}(\A - \lambda_1 \I)$}, consider the orthogonal decomposition of $\z=\z_1+\z_2$ such that  $\z_1\in\ra(\D)$ and $\z_2\in\Null(\D)$. {We note that {$\z_1,\z_2\neq\bf0$} since $\x \notin{\rm Range}(\A - \lambda_1\I)$, $\x^*\in{\rm Range}(\A - \lambda_1 \I)$ and $\alpha >0$.} Let \begin{equation}\label{eq:defc}
c=\frac{\d_{K+1} c_1 }{16\sqrt{\d_n^2+\d_{K+1}^2}}<\frac{c_1}{16}<\frac{1}{16},
\end{equation} where  $c_1$ is the constant in \eqref{eq:KLrange}.
We consider two subcases:
\begin{enumerate}
\item
Suppose that $\|\z_2\|<c\sqrt\gamma$.
 Then, we have
\begin{equation}\label{eq:betage}
\beta = \inprod{\D\z}{\z} = \inprod{\D\z_1}{\z_1}\ge\d_{K+1}\|\z_1\|^2\ge\d_{K+1}(1-c^{2})\gamma\ge\frac{\d_{K+1}}{4}\gamma.
\end{equation}
In our proof, the dimension of subspace ${\rm Range}(\D)$ needs to be taken into consideration. Specifically, we use different strategies to handle cases with $\dim(\ra(\D))\le2$ and $\dim(\ra(\D))\ge3$ (which implicitly requires $n\ge 4$). The case with $\dim(\ra(\D))\le2$ is less insightful and its proof is presented in Appendix \ref{app:dim2}.

\medskip

Now we focus on the case with $\dim(\ra(\D))\ge3$.
Let $\tilde \z=\z_1+\z_3$ with $\z_3\in\ra(\D)$, $\z_3\perp\s_1$, $\z_3\perp \z_1$,  $\|\z_3\|=\|\z_2\|$ and $\langle \z_3,\D\z_1\rangle\le0$. Note that $\z_3$, and consequently $\tilde\z$, is well-defined because $\dim(\ra(\D))\ge3$.
From $\tilde \z$, we can further construct $\tilde \x$, and the corresponding $\tilde \alpha$ and $\tilde \beta$ as in \eqref{eq:defabP}, in the following way:
\[
\tilde \x =  \x^* + \P \tilde\z, \quad
\tilde\alpha=\s^T\D\tilde\z  \quad \mbox{ and } \quad \tilde\beta=\tilde \z^T\D\tilde\z.
\]
Since $\norm{\s} = 1$, $\D\z_2=0$, $\langle \z_3,\D\z_1\rangle\le0$ and $\|\z_3\|=\|\z_2\|$,
we have
\begin{equation}\label{eq:alphatil}
|\tilde\alpha-\alpha|=|\s^T\D(\tilde\z-\z)|=|\s^T\D(\z_3-\z_2)|=|\s^T\D\z_3|\le\norm{\s}\norm{\D \z_3}\le \d_n\|\z_2\|,
\end{equation}
and
\begin{equation*}\label{eq:betatilub}
|\tilde\beta-\beta|=2\langle \z_3,\D\z_1\rangle+\langle \z_3,\D\z_3\rangle \le \d_n\|\z_3\|_2^2 = \d_n\|\z_2\|_2^2 \le \d_n c^2\gamma\overset{\eqref{eq:betage}}\le\frac{4\d_n c^2}{\d_{K+1}}\beta.
\end{equation*}
Meanwhile, we have from \eqref{eq:defc} that $4 \d_n c^2/\d_{K+1} \le 1/4$, and consequently,
\begin{equation}
\label{eq:betatilde}
\frac{1}{4}\beta \le (1 - \frac{4 \d_n c^2}{\d_{K+1}})\beta \le \tilde\beta \le (1 + \frac{4\d_n c^2}{\d_{K+1}}) \beta \le 2 \beta.
\end{equation}
Furthermore, it holds that  $\norm{\tilde \x - \x^*}^2 = \|\tilde\z\|^2=\|\z\|^2=\gamma$ and $\norm{\tilde \x}^2 = \|\s+\tilde\z\|^2=\|\s+\z\|^2 = \norm{\x}^2 = 1$.
Since $\tilde \x \in {\bf bd}(B)\cap {\rm Range}(\A - \lambda_1\I)$ and $\norm{\tilde \x - \x^*} = \sqrt{\gamma} \le \epsilon_1$, we know from \eqref{eq:KLrange} that
\begin{equation}\label{eq:consteasy}
 \frac{1}{2}{\rm dist}(-\nabla f(\tilde\x), N_B(\tilde\x)) = \|\D \tilde \z-(\tilde\alpha+\tilde\beta)(\s+\tilde \z)\|\ge c_{1}\sqrt{\d_{K+1}}\tilde\beta^{1/2}.
\end{equation}
Next, we bound the difference between ${\rm dist}(-\nabla f(\tilde\x), N_B(\tilde\x))/2$ and ${\rm dist}(-\nabla f( \x), N_B( \x))/2$ by
\begin{equation}\label{eq:tildeDD}
\begin{array}{lll}
&&\|[\D \tilde \z-(\tilde\alpha+\tilde\beta)(\s+\tilde \z)]-[\D \z-(\alpha+\beta)(\s+\z)]\|\\
&=&\|\D (\tilde \z-\z)-(\tilde\alpha-\alpha)\s-\tilde\alpha\tilde\z+\alpha\z-\tilde\beta(\s+\tilde \z)+\beta(\s+\z)\|\\
&\le &\|\D\z_3\|+|\tilde\alpha-\alpha|+|\tilde\alpha|\|\tilde\z\|+|\alpha|\|\z\|+\tilde\beta+\beta\\
&\le& 2\d_{n}\|\z_2\|+2\d_n\gamma+ 3\beta,
\end{array}
\end{equation}
where  %
the first inequality follows from  $\D(\tilde\z-\z)=\D(\z_3-\z_2)=\D\z_3$, $\|\s\|=1$ and $\|\s+\tilde\z\|=\|\s+\z\|=1$,
and the last inequality follows from
$\|\D\z_3\|\le \d_n\|\z_3\|=\d_n\|\z_2\|$, \eqref{eq:alphatil},
$|\alpha|\|\z\|=|\s^T\D\z|\|\z\|\le\d_n\|\z\|^2=\d_n\gamma$,
$|\tilde\alpha|\|\tilde\z\|=|\s^T\D\tilde\z|\|\tilde\z\|\le\d_n\|\tilde\z\|^2=\d_n\gamma$ and \eqref{eq:betatilde}.
The right-hand-side of \eqref{eq:tildeDD} can be further bounded by
\begin{equation*}
\begin{array}{lll}
2\d_{n}\|\z_2\|+2\d_n\gamma+ 3\beta&\overset{\eqref{eq:betage}}\le& 2\d_{n}\|\z_2\|+(3+\frac{8\d_n}{\d_{K+1}})\beta\\
&\overset{}\le& 2\d_{n}c\sqrt\gamma+(3+\frac{8\d_n}{\d_{K+1}})\beta\\
&\overset{\eqref{eq:defc}}\le&\frac{c_{1}}{8}\d_{K+1}\sqrt\gamma +(3+\frac{8\d_n}{\d_{K+1}})\beta\\
&\overset{\eqref{eq:betage}}\le&\frac{c_{1}}{4}\sqrt{\d_{K+1}}\beta^{1/2}+(3+\frac{8\d_n}{\d_{K+1}})\beta\\[2pt]
&{}=& \big(\frac{c_{1}}{4}\sqrt{\d_{K+1}} + (3 + \frac{8\d_n}{\d_{K+1}})\beta^{1/2}\big)\beta^{1/2}.
\end{array}
\end{equation*}
Since $\beta\to 0$ as $\x \to \x^*$, by reducing  $\epsilon_0$ if necessary, we have
\[
(3 + \frac{8\d_n}{\d_{K+1}})\beta^{1/2} \le \frac{c_1}{8} \sqrt{\d_{K+1}}
\]
and consequently,
\begin{equation}
\label{eq:bdc3}
2\d_{n}\|\z_2\|+2\d_n\gamma+ 3\beta \le \frac{3c_1}{8} \sqrt{\d_{K+1}}.
\end{equation}
Therefore, we know that
\begin{equation}\label{eq:hc2i2a}
\begin{array}{lll}
H(\x)&=&\|\D \z-(\alpha+\beta)(\s+\z)\|\\
&\ge& c_{1}\sqrt{\d_{K+1}}\tilde\beta^{1/2}-2\d_{n}\|\z_2\|-2\d_n\gamma-c_3\beta\\
&\ge&\frac{c_{1}}{2}\sqrt{\d_{K+1}}\beta^{1/2}-\frac{3c_{1}}{8}\sqrt{\d_{K+1}}\beta^{1/2} \\
&\ge&\frac{c_{1}}{8}\sqrt{\d_{K+1}}\beta^{1/2}, \\
\end{array}
\end{equation}
where the first inequality is due to \eqref{eq:consteasy} and \eqref{eq:tildeDD},
the second equality is due to \eqref{eq:betatilde} and \eqref{eq:bdc3}.
\item
On the other hand, suppose that  $\|\z_2\|\ge c\sqrt\gamma$. Then, $\|\z_1\|\le \sqrt{(1-c^2)\gamma}$.
Let  $\z'=\rho \z_1$  and {$\x'=\x^*+\P\z' = \s + \z'$}, where
$\rho=\gamma/\|\z_1\|^2\ge1/\sqrt{1-c^2}>1.$
Then we have
$$\norm{\x'} = \|\s+\z'\|^2=1+2\s^T\z'+\|\z'\|^2=1-\rho\gamma+\rho^2\|\z_1\|^2=1,$$
where the third equality is due to
$$2\s^T\z'=2\rho\s^T\z_1=2\rho\s^T\z=\rho(\|\s+\z\|^2-\|\s\|^2-\|\z\|^2) = \rho(1 - 1 -\gamma)=-\rho\gamma.$$
Similar as in \eqref{eq:defabP}, define $\alpha'$ and $\beta'$ corresponding to $\x'$ as  $\alpha'=\s^T\D\z'=\rho\alpha$ and $\beta'=\z'^T\D\z'=\rho^2\beta$.
Then, recalling the definition of $H$ in \eqref{eq:defH}, we have
\begin{eqnarray*}
\frac{H(\x')}{(\beta')^{3/2}}&=&\frac{\|\D\z'\|^2-\alpha'^2-2\alpha'\beta'-\beta'^2}{(\beta')^{3/2}}\\
&=&\frac{(\|\D\z\|^2-\alpha^2)/\rho-2\alpha\beta-\rho\beta^2}{\beta^{3/2}}\\
&=&\frac{(\|\D\z\|^2-\alpha^2-2\alpha\beta-\beta^2)/\rho+2(1/\rho-1)\alpha\beta+(1/\rho-\rho)\beta^2}{\beta^{3/2}}\\
&\le & \frac{H(\x)}{\rho\beta^{3/2}}+\frac{2(1/\rho-1)\alpha\beta}{\beta^{3/2}},
\end{eqnarray*}
where the last inequality is due to $1/\rho-\rho<0$.
Note that since $\|\x'\|=1$, it holds that $H(\x')=\|\D\z'\|-(\alpha'+\beta')^2=\|\D\z'-(\alpha'+\beta')(\x')\|^2\ge0$.
Hence we have
\begin{equation}\label{eq:hc2i2b}
\frac{H(\x)}{\beta^{3/2}}\ge\frac{2\rho(1-1/\rho)\alpha\beta}{\beta^{3/2}}\overset{\eqref{eq:algebe}}\ge\frac{\rho-1}{2}\sqrt{\d_{K+1}}\ge\frac{1/\sqrt{1-c^2}-1}{2}\sqrt{\d_{K+1}}.
\end{equation}
\end{enumerate}
Combining (\ref{eq:defabP}--\ref{eq:distandH}), \eqref{eq:hc2ii1}, \eqref{eq:alphage0},  \eqref{eq:hc2i2a} and \eqref{eq:hc2i2b} (and \eqref{eq:hc2i2adim2} in Appendix \ref{app:dim2} for the case with ${\rm dim}({\rm Range}(\D)) \le 2$) with Lemma \ref{lem:interiorKL}, and noting that, {by reducing $\epsilon_0$ if necessary}, $\beta^{3/4}\le\beta^{1/2}$, we know that \eqref{eq:hardKL2i} holds with some positive constant $\tau$ and sufficient small $\epsilon >0$.
\end{proof}

We summarize the results of Theorems \ref{thm:kl-convex} and Lemmas \ref{lem:interiorKL}--\ref{lem:hard2i}  in the following theorem.
\begin{thm}
\label{cor}
There exist some constant $\tau_{KL} >0$ and sufficiently small $\epsilon >0$ such that the KL inequality holds
\begin{equation}
\label{eq:KLgeneral}
\big(f(\x)-f^*+\delta_{B}(\x)\big)^\varrho\le\tau_{KL}{\rm dist}(-\nabla f(\x), N_B(\x)), \quad \forall\, \x\in B(\x^*,\epsilon),
\end{equation}
where $\varrho=\left\{
\begin{aligned}
& {3/4}, \mbox{ for the {TRS-ill case} \eqref{eq:con},} \\
& {1/2}, \mbox{ otherwise.}
\end{aligned}
\right.$
\end{thm}
The above theorem in fact shows us that the KL exponent $3/4$ holds for all cases, due to the fact that the function value $f(\x)$ is bounded in the unit ball and the inequality  $t^{3/4}\le t^{1/2}$ for all $0\le t\le 1$.

\section{Convergence analysis of first order methods}
Recently,  Beck and Vaisbourd \cite{beck2018globally} demonstrated that with  a proper initialization, {projected gradient methods} (in fact, more general first order conic methods) for solving the TRS converge to a global optimal solution.
However, the rate of convergence for these algorithms are not studied in their paper.
Meanwhile, it is well known that projected gradient methods converge to a stationary point in  rate  $O(1/\sqrt{k})$  for {minimizing a} nonconvex smooth function with global Lipschitz continuous gradient over a closed convex set with $k$ being the iteration index; see, e.g., \cite[Chapter 1]{nesterov1998introductory} for unconstrained minimization and \cite[Chapter 10]{beck2017first} for general composite minimization.
Hence, a straightforward conclusion will be that  {projected gradient methods} for the TRS achieve a sublinear iteration complexity of $O(1/\sqrt{k})$.
Here, we improve this result by showing that the local convergence rate of {projected gradient methods} for solving the TRS can be improved to at least $O(1/k^2)$. In fact, in most cases, it even enjoys local linear convergence. As one can observe in the subsequent analysis in this section, the cornerstone for these superior improvements is the obtained KL inequality for the TRS.

Let $\Pi_B:\R^n \to \R^n$ be the Euclidean projector onto the unit norm ball $B$, i.e., for any $\z\in\R^n$, \[\Pi_B(\z)=\left\{
\begin{aligned}
& {\frac{\z}{\|\z\|}}, &\mbox{ if } \norm{\z}  \ge 1, \\
& \z, &\mbox{ otherwise.}
\end{aligned}
\right.\]  Our main contribution in this section is summarized in the following theorem for constant step size projected gradient methods. Note that the Lipschitz constant of the gradient of the objective function in the TRS is $L = 2\|\A\|_2$ and the step size we will use is $t\in(0,\frac{2}{L})$, i.e., $t\in(0,\frac{1}{\norm{\A}_2})$.
Specifically, we show that {projected gradient methods} achieve a locally sublinear rate of $O(1/k^2)$ in the \emph{TRS-ill case}; otherwise, the local convergence rate can be further improved to linear.

\begin{thm}[Constant step size projected gradient methods]
\label{thm:rate}
Let the step size $t\in(0, 2/L)$, i.e., $t\in(0,1/\norm{\A}_2)$. Suppose for all $k\ge 0$ that  $\x_{k+1}=\Pi_B(\x_k-t\nabla f(\x_k))$ and the sequence $\{\x_k\}$ converge to an optimal solution $x^*$ to the TRS $\rm(P_0)$.
Then, there exists a sufficiently large positive integer $N$ such that
$\{\x_k\}_{k\ge N} \subseteq B(\x^*,\epsilon)$ where $\epsilon >0$ is the same constant in  Theorem  \ref{cor}.

Let $M=\frac{\tau_{KL} (2\|A\|_2+1/t)}{\sqrt{1/t-\|A\|_2}}$ with $\tau_{KL} >0$ being the constant in Theorem \ref{cor}. Then, in the \emph{TRS-ill case} \eqref{eq:con}, it holds that
\begin{equation}
\label{eq:rate1}
f(\x_k)-f^*\le\frac{1}{\left(\frac{k-N}{M^2(2+\frac{3}{2M^2}\sqrt{f(x_N)-f^*})}+\frac{1}{\sqrt{f(x_N)-f^*}}\right)^2}, \quad \forall\, k\ge N;
\end{equation}
otherwise, it holds that
\begin{equation*}
\label{eq:rate2}
f(\x_k)-f^*\le\left(\frac{M^2}{M^2 + 1}\right)^k(f(\x_N)-f^*), \quad \forall\, k\ge N.
\end{equation*}
\end{thm}
\begin{proof}
Since $\x_k \to \x^*$ as $k\to \infty$, for the given constant $\epsilon >0$ in Theorem \ref{cor}, there exists $N>0$ such that
$\x_k\in B(\x^*,\epsilon)$ for all $k\ge N$.

For all $k\ge 0$, one can rewrite the updating rule $\x_{k+1}=\Pi_B\big(\x_k-t\nabla f(\x_k)\big)$ in the following manner
\[\x_{k+1}=\argmin_\u \left\{ \nabla f(\x_k)^T(\u-\x_k)+\frac{1}{2t}\|\u-\x_k\|^2\right\}+\delta_B(\u),\]
whose optimality condition asserts:
\[ 0\in  \x_{k+1}-(\x_k-t\nabla f(\x_k)) +N_B(\x_{k+1})\] with $N_B(\x_{k+1})$ being the normal cone of $B$ at $\x_{k+1}$. Denote for all $k\ge 0$, $\v_{k+1}= - (\x_{k+1} - \x_k)/t - \nabla f(\x_k)$. Then, it holds for all $k\ge 0$ that $\v_{k+1} \in N_B(\x_{k+1})$ and
\begin{equation}
\label{eq:con1}
\|\v_{k+1}+\nabla f(\x_k)\|= \|\x_{k+1}-\x_k\|/t.
\end{equation}
Since $1/t > \norm{\A}_2=L/2$, from Lemma 10.4 in \cite{beck2017first}, we know that
\begin{equation}
\label{eq:con2}
f(\x_k)-f(\x_{k+1})\ge ({1/t-\|\A\|_2})\|\x_k-\x_{k+1}\|^2, \quad \forall\, k\ge 0.
\end{equation}
Meanwhile, from the KL inequality \eqref{eq:KLgeneral}, it holds
for all $k\ge N$ that
\begin{eqnarray*}
(f(\x_{k+1})-f^*)^{\varrho}&\le & \tau_{KL} {\rm dist}(-\nabla f(\x_{k+1}), N_B(\x_{k+1}))\\
&\le& \tau_{KL} \|\nabla f(\x_{k+1})+\v_{k+1}\| \\
&=&\tau_{KL} \|\nabla f(\x_{k+1})-\nabla f(\x_k)+\nabla f(\x_k)+\v_{k+1}\|\\
&\le& \tau_{KL} \big(\|\nabla f(\x_{k+1})-\nabla f(\x_k)\|+\|\nabla f(\x_k)+\v_{k+1}\|\big),
\end{eqnarray*}
where $\varrho$ is the KL exponent and $\tau_{KL} >0$ is the constant in Theorem \ref{cor}.
The above inequality, together with the Lipschitz continuity of $\nabla f$ (note that the Lipschitz constant is $L=2\|\A\|_2$) and \eqref{eq:con1}, gives
\[
(f(\x_{k+1})-f^*)^{\varrho}\le \tau_{KL} (2\|\A\|_2+1/t)\|\x_k-\x_{k+1}\|, \quad \forall\, k\ge N.
\]
Substituting this to \eqref{eq:con2} implies
\[
(f(\x_{k+1})-f^*)^{\varrho}\le\frac{\tau_{KL} (2\|\A\|_2+1/t)}{\sqrt{1/t-\|\A\|_2}}(f(\x_{k})-f(\x_{k+1}))^\frac{1}{2}, \quad\forall\, k\ge N.
\]
Defining $r_k=f(\x_{k})-f^*$ and $M=\frac{\tau_{KL} (2\|\A\|_2+1/t)}{\sqrt{1/t-\|\A\|_2}}$, we have
\begin{equation}
\label{eq:rk}
r_{k+1}^{\varrho}\le M(r_{k}-r_{k+1})^\frac{1}{2}, \quad \forall\, k\ge N.
\end{equation}

We divide our discussions into two cases:
\begin{itemize}
        \item In the \emph{TRS-ill case} \eqref{eq:con}, from Theorem \ref{cor} and \eqref{eq:rk}, we have $r_{k+1}^\frac{3}{2}\le M^2(r_k-r_{k+1})$.
        Hence, for all $k\ge N$, we have {from \eqref{eq:con2}} that $r_{k+1} \le r_k$ and
              \begin{eqnarray*}
        \frac{1}{\sqrt{r_{k+1}}}-\frac{1}{\sqrt{r_{k}}}&\ge&\frac{1}{\sqrt{r_{k+1}}}-\frac{1}{\sqrt{r_{k+1}+\frac{1}{M^2}r_{k+1}^{3/2}}}\\
        &=&\frac{\sqrt{r_{k+1}+\frac{1}{M^2}r_{k+1}^{3/2}}-\sqrt{r_{k+1}}}{\sqrt{r_{k+1}}\sqrt{r_{k+1}+\frac{1}{M^2}r_{k+1}^{3/2}}}\\
        &=&\frac{r_{k+1}^{3/2}}{M^2\sqrt{r_{k+1}}\sqrt{r_{k+1}+\frac{1}{M^2}r_{k+1}^{3/2}}(\sqrt{r_{k+1}+\frac{1}{M^2}r_{k+1}^{3/2}}+\sqrt{r_{k+1}})}\\
        &=&\frac{1}{M^2\sqrt{1+\frac{1}{M^2}\sqrt{r_{k+1}}}(\sqrt{1+\frac{1}{M^2}\sqrt{r_{k+1}}}+1)}\\
        &=&\frac{1}{M^2(\sqrt{1+\frac{1}{M^2}\sqrt{r_{k+1}}}+1+\frac{1}{M^2}\sqrt{r_{k+1}})}\\
        &\ge&\frac{1}{M^2(2+\frac{3}{2M^2}\sqrt{r_{k+1}})}\\
        &\ge&\frac{1}{M^2(2+\frac{3}{2M^2}\sqrt{r_{N}})},
        \end{eqnarray*}
        where the second inequality follows from $\sqrt{1+\frac{1}{M^2}\sqrt{r_{k+1}}}\le 1+\frac{1}{2M^2}\sqrt{r_{k+1}}$.
Hence, for all $k\ge N$,  we have
\[\frac{1}{\sqrt{r_k}}\ge\frac{k-N}{M^2(2+\frac{3}{2M^2}\sqrt{r_{N}})}+\frac{1}{\sqrt{r_N}}
\]
and thus
\[ r_k\le\frac{1}{\left(\frac{k-N}{M^2(2+\frac{3}{2M^2}\sqrt{r_{N}})}+\frac{1}{\sqrt{r_N}}\right)^2}\,.\]

        \item Otherwise, from Theorem \ref{cor} and \eqref{eq:rk}, we have $r_{k+1}\le M^2 (r_k-r_{k+1})$ for all $k\ge N$. This implies \[
        r_{k}\le \left(\frac{M^2}{M^2+1}\right)^{k - N}r_N,\quad \forall\, k\ge N.\]
\end{itemize}

We thus complete the proof for the theorem.
\end{proof}
We remark here
that the sublinear convergence rate \eqref{eq:rate1} holds in all cases as the KL inequality  holds with exponent $3/4$ for all cases. Note that the assumption in Theorem \ref{thm:rate} that {projected gradient methods} converge to a global optimal solution is not restrictive at all. As is mentioned in the introduction  and the beginning of this section, the assumption can be guaranteed as long as the starting point of the projection gradient method is properly chosen \cite{beck2018globally}.
Moreover, it is also noted in \cite{beck2018globally} that the initial point can be obtained without much difficulty.
Although the iteration complexity results derived in Theorem \ref{thm:rate} only holds locally around the optimal solution $\x^*$, by using similar ideas in \cite{han2017linear}, one can directly extend these results to a global version.
{To see this,  let $N$ be the positive integer in Theorem \ref{thm:rate}.
Note from the proof of Theorem \ref{thm:rate} that $\{f(\x_k)\}_{k\ge 0}$ is non-increasing. Then, we have that for all $1\le k\le N$,
\[
0 \le f(\x_k)-f^* \le  f(\x_0) - f^* \le
\min \left\{\frac{N^2( f(\x_0) - f^*)}{k^2} , (f(\x_0) - f^*) (\frac{M^2}{M^2 + 1})^{k-N} \right\}.
\]
Hence {it follows from Theorem \ref{thm:rate}} there exist constants $C_1, C_2 >0$
such that
\begin{equation*}
\label{eq:complexity_global}
f(\x_k) - f(\x^*) \le
\left\{
\begin{array}{ll}
 \frac{C_1}{k^2}, &\text{for the \emph{TRS-ill case} \eqref{eq:con}},\\
 C_2\, (\frac{M^2}{M^2 + 1})^{k}, &\text{otherwise,}
\end{array}
\right.\quad \forall\, k\ge 1.
\end{equation*}
{Note that constants $C_1,C_2$ depend on the choice of the initial point and can be hard to estimate explicitly.}
In light of Theorem \ref{thm:main}, we can further derive the convergence rate associated with $\dt(\x_k,S_0)$:
\begin{equation*}
\label{eq:complexity_global_dist}
\dt(\x_k,S_0)\le\left\{
\begin{array}{ll}
 \frac{\tau_{EB}C_1^{1/4}}{\sqrt k}, &\text{for the \emph{TRS-ill case} \eqref{eq:con}},\\
 \tau_{EB} \sqrt{C_2}\,(\frac{M^2}{M^2 + 1})^{k/2}, &\text{otherwise,}
\end{array}
\right. \quad \forall\, k\ge 1,
\end{equation*}
where $\tau_{EB} >0$ is the error bound modulus in Theorem \ref{thm:main}. We also note that the projected gradient method with backtracking line search can also be analyzed in a similar way, as  \eqref{eq:con2} still holds with a little more conservative constant, and thus we omit its analysis for simplicity.
\begin{rem}
It would be interesting to compare Theorem \ref{thm:rate} with the results obtained in \cite{zhang2017generalized} and \cite{carmon2018analysis}, which demonstrated that the generalized Lanczos trust-region (GLTR)\ method has a linear convergence rate for the easy case.
While we have proved that {projected gradient methods} converge sublinearly at rate $O(1/k^2)$ for the \emph{TRS-ill case}, and linearly otherwise (including the easy case).
\end{rem}
\begin{rem}
Based on the established KL inequality for the TRS in Theorem \ref{cor}, the same order of the convergence rate for the gap $f(\x_k)-f(\x^*)$ and $\dt(\x_k,S_0)$ can also be obtained by using the arguments in \cite{frankel2015splitting}. However,  our proof here is simpler and has an explicit dependence of the constants.
\end{rem}

\section{Conclusion}
In this paper, we conducted a thorough analysis about H\"olderian error bounds   and the KL inequality associated with the TRS. Specifically,
 we showed that for the TRS, a H\"olderian error bound holds with modulus 1/4 and  the KL inequality  holds with exponent 3/4.
 Moreover, we demonstrated that the H\"olderian error bound modulus and the KL exponent in fact are both 1/2 unless in the \emph{TRS-ill case} \eqref{eq:con}.
 Given these results, we further proved that {projected gradient methods} for solving the TRS enjoy a local sublinear convergence rate $O(1/k^2)$, which can be further improved to a linear rate unless in the \emph{TRS-ill case}.

\section*{Acknowledgements}
We would like to thank Professor Jong-Shi Pang at University of Southern California for his helpful comments on an early version of this paper.

\section{Appendix}

\subsection{{Complementary} proof for Lemma \ref{lem:hard2ii}}
Here, we provide a supporting lemma that helps prove Lemma \ref{lem:hard2ii}.
With the same notation $\alpha,\beta$ in \eqref{eq:defabP} and $H(\x)$ in \eqref{eq:defH}, recall that $\|\s\|=\|\P^T\x^*\|=1$ and $\{\d_i\}_{i=1}^n$ in \eqref{eq:kD}.
We have the following lemma.
\begin{lem}
        \label{lem:funcPsi}
        It holds that for all $\z\in\R^n$
        \[
        H(\x)= \sum_{i=1}^K \s_i^2\sum_{j= K+1}^{n}  (\d_j \z_j)^2
         + \sum_{K+1\le i < j \le n}(\d_i\z_i \s_j - \d_j\z_j \s_i)^2 - 2\alpha\beta-\beta^2.
        \]
        Moreover, if  $\s\not\perp {\rm Null}(\D)$, it holds that
        \begin{equation}
        \label{eq:psizs}
        H(\x)\ge \frac{1}{2}\d_{K+1}(\sum_{j=1}^K \s_j^2) \beta
        \end{equation}
        for all $\z$ with sufficiently small $\norm{\z}$.
\end{lem}
\begin{proof}
By some calculations, we see that
\begin{align*}
  H(\x) ={}& \sum_{i=1}^n (\d_i\z_i)^2 - \big(\sum_{i=1}^{n} \s_i \d_i \z_i\big)^2 - 2\alpha\beta-\beta^2\\
={}&\sum_{i=1}^n \sum_{j=1}^n \s_j^2 (\d_i \z_i)^2 - \big(\sum_{i=1}^{n} \s_i \d_i \z_i\big)^2 - 2\alpha\beta-\beta^2 \\
={}& \sum_{1\le i <j \le n}(\d_i\z_i \s_j - \d_j\z_j \s_i)^2 - 2\alpha\beta-\beta^2\\
={}&\sum_{i=1}^K \sum_{j=K+1}^{n}  (\d_i\z_i \s_j - \d_j\z_j \s_i)^2
+ \sum_{i=1}^K \sum_{j=i+1}^{K}  (\d_i\z_i \s_j - \d_j\z_j \s_i)^2\\
&+ \sum_{K+1\le i < j \le n}(\d_i\z_i \s_j - \d_j\z_j \s_i)^2 - 2\alpha\beta-\beta^2\\
={}& \sum_{i=1}^K \s_i^2\sum_{j=K+1}^{n}  (\d_j \z_j)^2
+ \sum_{K+1\le i < j \le n}(\d_i\z_i \s_j - \d_j\z_j \s_i)^2 - 2\alpha\beta-\beta^2,
\end{align*}
where the second equation holds since $\norm{\s} = 1$,
the third equation follows from the fact that $\sum_{1\le i,j\le n} a_i^2 b_j^2 - (\sum_{1\le i \le n} a_ib_i)^2 = \sum_{1\le i <j \le n} (a_ib_j - a_jb_i)^2$,
and the last equation is due to $\d_i =0$ for $i=1,\ldots,K$.

 Since $\sum_{i=K+1}^n\d_i^2 \z_i^2 \ge \sum_{i=K+1}^n \d_{K+1} \d_i \z_i^2 = \d_{K+1}\beta$, we obtain that
 \[
H(\x)\ge (\d_{K+1} \sum_{j=1}^K\s_j^2 - 2\alpha-\beta)\beta, \quad \forall\, \z \in \R^n.
 \]
If, further, $\s\not\perp {\rm Null}(\D)$, together with $\d_{K+1} >0$, we know that
 $\d_{K+1} \sum_{j=1}^K \s_j^2 >0$.
 This together with the fact that $\alpha\to 0$ and $\beta\to 0$ as $\z\to {\bf 0}$, implies that \eqref{eq:psizs} holds when $\norm{\z}$ is sufficiently small. We thus complete the proof.
\end{proof}

\subsection{Proof of Lemma \ref{lem:KLld}}
\begin{proof}
	Recall the spectral decomposition of $\A$ and the definition of the diagonal matrix $\D$ as:
	\[
	\A = \P{\mathbf \Lambda}\P^T = \P {\rm Diag}(\lambda)\P^T, \quad
	\A - \lambda_1\I = \P \D \P^T, \quad \D = {\mathbf \Lambda} -\lambda_1 \I,
	\]
{where $\P \P^T = \P^T\P = \I$.}
	Define linear operators $\mathcal{L}_{K} :\R^{n} \to \R^{K}$ and $\mathcal{U}_{K}:\R^n \to \R^{n-K}$ as
	\[
	\mathcal{L}_{K}(\x) = [\x_{1}, \ldots, \x_{K}]^T, \quad \mathcal U_{K}(\x) = [\x_{K+1},\ldots, \x_n]^T, \quad \forall\, \x\in\R^n.
	\]
	Let $\M = {\rm Diag}(\mathcal U_{K}(\lambda))$ where $\lambda \in \R^n$ is {the vector formed by all the eigenvalues} of $\A$ in ascending order. Define function $q$ as
	\[
	q(\z) := \inprod{\z}{\M\z} - 2\inprod{\mathcal U_{K}(\P^T \b)}{\z}, \quad \forall \z \in \R^{n-K}.
	\]
 Consider the following optimization problem
	\begin{equation}
	\label{prob:ming}
	\min \left\{
	q(\z) \mid \z \in \widehat{B} \right\} \mbox{ with } \widehat{B} : = \left\{ \z \in \R^{n - K} \mid \norm{\z} \le 1 \right\}.
	\end{equation}
	We show next that $\mathcal U_{K}(\P^T \x^*)$ is an optimal solution to \eqref{prob:ming}.
Since  $\b, \x^*\in {\rm Range}(\A - \lambda_1 \I)$, we know that $\P^T \b, \P^T \x^* \in {\rm Range}({\mathbf \Lambda} - \lambda_1\I)$.
Note that since the first $K$ diagonal entries of the diagonal matrix ${\mathbf \Lambda} - \lambda_1\I$ are zeros, we have $\mathcal L_{K}(\P^T \x^*)) =\mathcal L_{K}(\P^T\b) = \bf0$ and $\norm{\mathcal U_{K}(\P^T \x^*)} = 1$.
Moreover, since $(\A -\lambda_1 \I)\x^* = \b$, we have  $({\mathbf \Lambda} - \lambda_1\I)(\P^T \x^*) = \P^T \b $.
Therefore, $ (\M - \lambda_1 \I) \mathcal U_K(\P^T\x^*) = \mathcal U_{K}(\P^T\b)$, i.e., $\mathcal U_{K}(\P^T \x^*)$ and $-\lambda_1$ solve the optimality conditions (see the KKT conditions in Section \ref{sec:pre}) associated with the smaller dimensional TRS \eqref{prob:ming}.
	Since $\lambda_{K+1}$ is the smallest eigenvalue of $\M$ and {an optimal Lagrangian multiplier is} $-\lambda_1> -\lambda_{K+1}$, we know that problem  \eqref{prob:ming} falls into the easy case or the hard case 1. Hence, by Lemma \ref{lem:KLeasyhard1}, we have that there exist $\epsilon_2, \tau >0$ such that for all $\z \in B({\cal U}_K(\P^T\x^*), \epsilon_2)$,
	\begin{equation}
	\label{eq:klqz}
	[q(\z) - q(\mathcal U_k(\P^T\x^*)) + \delta_{\widehat{B}}(\z)]^{1/2} \le \tau {\rm dist}(-\nabla q(\z), N_{\widehat{B}}(\z)).
	\end{equation}
	Note that for all $\x \in {\rm Range}(\A - \lambda_1\I)$, it holds that $\mathcal L_K(\P^T\x) = \bf0$ and $\norm{{\cal U}_K(\P^T \x)} = \norm{\P^T \x} = \norm{\x}$.
	Then, for all $\x\in {\rm Range}(\A - \lambda_1\I)$ with $\norm{\x} < 1$, we know that $\norm{{\cal U}_K(\P^T\x)} < 1$, $N_{\widehat{B}}({\cal U}_K(\P^T\x)) =\bf 0$, and
	\begin{equation}
	\label{eq:distqin}
	\begin{aligned}
	{\rm dist}(-\nabla q(\mathcal U_K(\P^T \x)), N_{\widehat{B}}(\mathcal U_K(\P^T \x))) ={}& \norm{\nabla q(\mathcal U_K(\P^T \x))} = \norm{2(\M\mathcal U_K(\P^T \x) - \mathcal U_K(\P^T \b))} \\[2pt]
	={}&\norm{2({\mathbf \Lambda} \P^T\x - \P^T\b)} = \norm{2(\P{\mathbf \Lambda}\P^T\x -\b} \\[2pt]
	={}& \norm{2(\A\x - \b)} = {\rm dist}(-\nabla f(\x), N_B(\x)).
	\end{aligned}
	\end{equation}
	Meanwhile, for all $\x \in {\rm Range}(\A - \lambda_1\I)$ with $\norm{\x} = 1$, it holds  $\norm{{\cal U}_K(\P^T \x)} = 1$ and
	\[
	\norm{\mathcal U_K(\P^T\x) - \mathcal U_K(\P^T \x^*)} = \norm{\P^T\x - \P^T \x^*} =
	\norm{\x - \x^*}.
	\]
{Then for all $\x \in {\bf bd}(B) \cap {\rm Range}(\A - \lambda_1\I) $ satisfying $\norm{\x - \x^*}\le \epsilon: = \min\left\{\epsilon_0,\epsilon_2 \right\}$, we further have}
	\begin{equation}
	\label{eq:distqdistf}
	\begin{aligned}
	&{\rm dist}(-\nabla q(\mathcal U_K(\P^T \x)), N_{\widehat{B}}(\mathcal U_K(\P^T \x))) \\
	={}& \norm{2(\M\mathcal U_K(\P^T \x) - \mathcal U_K(\P^T \b)) +
		\inprod{-2(\M\mathcal  U_K(\P^T \x) -\mathcal  U_K(\P^T \b))}{\mathcal U_K(\P^T \x)}\mathcal U_K(\P^T \x) }\\
	={}&\norm{2({\mathbf \Lambda} \P^T \x - \P^T\b) + \inprod{-2({\mathbf \Lambda} \P^T \x - \P^T\b)}{\P^T \x} \P^T \x}\\
	={}&\norm{2(\P {\mathbf \Lambda} \P^T\x - \b) + \inprod{-2(\P{\mathbf \Lambda} \P^T\x - \b)}{\x} \x} \\
	={}&\norm{2(\A\x - \b) + \inprod{-2(\A\x - \b)}{\x}\x}\\
	={}&{\rm dist}(-\nabla f(\x), N_B(\x)),
	\end{aligned}
	\end{equation}
	where $\epsilon_0$ is given in \eqref{eq:eps0}, the first and last equations follow from \eqref{eq:distnorm1}, \eqref{eq:nux}, \eqref{eq:eps0} and
$$\inprod{-2(\M\mathcal  U_K(\P^T \x)-\mathcal  U_K(\P^T \b))}{\mathcal U_K(\P^T \x)}=\inprod{-2(\A\x - \b)}{\x}>0,$$
and the third equation holds since $\P$ is an orthogonal matrix.
	Since $q(\mathcal U_k(\P^T\x^*)) = f(\x^*) = f^*$ and
	$q(\mathcal U_k(\P^T \x)) = f(\x)$ for all $\x \in {\rm Range}(\A - \lambda_1 \I)$, we know from \eqref{eq:klqz}, \eqref{eq:distqin} and \eqref{eq:distqdistf} that
	\begin{equation*}\label{eq:hd2ieasy}
	\big(f(\x) - f(\x^*) + \delta_B(\x)\big)^{1/2} \le \tau {\rm dist}(-\nabla f(\x), N_B(\x)),
	\end{equation*}
	for all
	$\x\in B(\x^*,\epsilon) \cap {\rm Range}(\A - \lambda_1 \I)$.
\end{proof}

\subsection{Complementary proof for Lemma \ref{lem:hard2i}}\label{app:dim2}
As a complementary proof for Lemma \ref{lem:hard2i}, we use the same notation as in the main proof. Here, we consider  {{\bf Case II} }  with $\dim(\ra(\D)) \le 2$.  {Recall that we are in case where}  $\|\s\| = \|\s+\z\|=1$, $\alpha = \inprod{\s}{\D\z} > 0$,  $\|\z\|^2=\gamma < 1/4$, $\z = \z_1 + \z_2$ with $\z_1 \in \ra(\D)$ and $\z_2 \in {\rm Null}(\D)$ and $\norm{\z_2}\le c \sqrt{\gamma}$ with constant $c \le 1/16$ given in \eqref{eq:KLrange}.
It can also be verified that $\inprod{\s}{\z_1} = \inprod{\s}{\z} = -\gamma/2$.

We first claim that $\dim(\ra(\D))>1$. Indeed, if $\dim(\ra(\D))=1$, since $\s \in \ra(\D)$, it holds that $\ra(\D)={\rm span}\{\s\}$, i.e., the space spanned by $\{\s\}$.
Since $\z_1\in\ra(\D)$ and $\inprod{\s}{\z_1}  = -\gamma/2$, we have $\z_1 = -\gamma\s/2$.
Hence,
\[
\inprod{\s}{\D\z} = \inprod{\s}{\D\z_1} = -\frac{\gamma}{2} \inprod{\s}{\D\s} \overset{\eqref{eq:kD}}< 0.
\]
However, this contradicts  our assumption   that $\alpha = \inprod{\s}{\D\z} > 0$.
Hence, $\dim(\ra(\D)) = 2$.
Therefore, \eqref{eq:kD} reduces to
$$\d_n \ge \d_{n-1}>0=\d_{n-2}=\ldots=\d_1, \quad \mbox{  i.e., } \quad K = n-2.$$
Let $\e_i$, $i=1,\ldots,n$, be the standard basis of $\R^n$.
Then, we can represent $\s$ and $\z$ by the linear combinations of $\e_n$ and $\e_{n-1}$ as  $\s=u_1\e_{n-1}+u_2\e_{n}$ and $\z_1=w_1\e_{n-1}+w_2\e_{n}$, respectively.
Since $\norm{\s} = 1, \norm{\z}^2 = \gamma$ and $\norm{\z_2} \le c\sqrt{\gamma}$, we know that
\begin{equation}
\label{eq:z1norm}
u_1^2 + u_2^2 = 1, \quad \mbox{and} \quad \norm{\z_1}^2 = w_1^2 + w_2^2 = \gamma - \norm{\z_2}^2 \ge (1 - c^2)\gamma.
\end{equation}

We discuss two cases here, i.e., cases with $(u_2\d_nw_2)(u_1\d_{n-1}w_1)> 0$ and $(u_2\d_nw_2)(u_1\d_{n-1}w_1) \le 0$.
Suppose that $(u_2\d_nw_2)(u_1\d_{n-1}w_1)> 0$. By reducing  $\epsilon_0$ if necessary, we can assume without
loss of generality that
\begin{equation}
\label{eq:gammaeps1}
\sqrt{\gamma} \le \epsilon_1 = \min(\epsilon_0, 1/2) <  \min (|u_1|, |u_2|).
\end{equation}
Since $\s^T\z=u_1w_1+u_2w_2=-\gamma/2<0$, $\d_{n-1}>0$ and $\d_n>0$, we have $(u_2\d_nw_2)<0$ and $(u_1\d_{n-1}w_1)<0$.
From \eqref{eq:z1norm} and $c\le 1/16$, we know that
\[
|w_1| + |w_2| \ge \sqrt{w_1^2 + w_2^2} \ge \frac{1}{2} \sqrt{\gamma}.
\]
This further implies that
\begin{eqnarray*}
\s^T\z = u_2w_2+u_1w_1 &=& -|u_2||w_2| - |u_1||w_1|\\
	&\le& -\min(|u_1|,|u_2|)(|w_2|+|w_1|)\\
	&\le&-\frac{\min(|u_1|,|u_2|)\sqrt\gamma}{2} \\
&\overset{\eqref{eq:gammaeps1}}<& -\frac{\gamma}{2}.
\end{eqnarray*}
which contradicts the fact that $\s^T\z=-\gamma/2$.

Hence, we only need to consider the case with
$(u_2\d_nw_2)(u_1\d_{n-1}w_1)\le 0$. Then, it holds that
\begin{eqnarray}\label{eq:s1s2neq0}
\begin{array}{lll}
\|\D\z\|^2-\alpha^2&=&\d_n^2w_2^2+\d_{n-1}^2w_1^2-(u_2\d_nw_2+u_1\d_{n-1}w_1)^2\\
&\ge&(1-u_2^2)\d_n^2w_2^2+(1-u_1^2)\d_{n-1}^2w_1^2\\
&=&u_1^2 \d_{n}^2w_2^2+ u_2^2 \d_{n-1}^2w_1^2.
\end{array}
\end{eqnarray}
Suppose first that  {$\min(u_1^2,u_2^2)= 0$}. Without loss of generality, we can assume $u_1 = 0$, and thus $|u_2| = 1$. Since $\inprod{\s}{\z_1} = -\gamma/2 = u_2w_2$, we know that $w_2^2 = \gamma^2/4$, which, together with \eqref{eq:z1norm} and the fact that $\gamma \le 1/4$ and $c \le 1/16$, implies that
\begin{equation}
\label{eq:w12}
w_1^2 \ge (1 - c^2)\gamma - \gamma^2/4 \ge (1 - 1/16^2 - 1/16)\gamma \ge \frac{7}{8}\gamma.
\end{equation}
Meanwhile, it holds that
\[
\beta= \inprod{\z}{\D\z} = \inprod{\z_1}{\D\z_1} \le \d_n \norm{\z_1}^2 \le \d_n \gamma,
\]
which, together with $\eqref{eq:s1s2neq0}$ and $\eqref{eq:w12}$, implies that
\begin{equation}
\label{eq:u1u20alpha}
\|\D\z\|^2-\alpha^2 \ge \frac{7}{8}\d_{n-1}^2 \gamma
\ge  \frac{7\d_{n-1}^2}{8\d_n} \d_n \gamma \ge \frac{7\d_{n-1}^2}{8\d_n} \beta.
\end{equation}
Suppose now that  {$\min(u_1^2,u_2^2)\neq 0$}. Then, from \eqref{eq:s1s2neq0}, we know that
\begin{equation}
\label{eq:u1u2n0alpha}
\begin{array}{lll}
\|\D\z\|^2-\alpha^2 &\ge{}& \min(u_1^2, u_2^2) (\d_n^2 w_2^2 + \d_{n-1}^2 w_1^2 ) \\
&\ge{} &  \min(u_1^2, u_2^2) \d_{n-1} (\d_n w_2^2 + \d_{n-1} w_1^2 ) \\
&=& \min(u_1^2, u_2^2) \d_{n-1} \beta.
\end{array}
\end{equation}
Combining \eqref{eq:u1u20alpha} and \eqref{eq:u1u2n0alpha}, we conclude that if $\dim(\ra(\D))=2$, it holds that
\begin{equation*}\label{eq:dim2}
\|\D\z\|^2-\alpha^2\ge c_2\beta,
\end{equation*}
where $c_2=\min\{\frac{7\d_{n-1}^2}{8\d_n},\min(u_1^2, u_2^2) \d_{n-1}\}$.
Since $\alpha\to0$ and $\beta\to0$ as $\gamma\to 0$ (or equivalently, $\x \to \x^*$), by reducing  $\epsilon_0$ if necessary, we know in {{\bf Case II}} if $\dim(\ra(\D))=2$ that
\begin{equation}
\label{eq:hc2i2adim2}
\frac{1}{4}{\rm dist}^2(-\nabla f(\x), N_B(\x)) = H(\x)=\|\D\z\|^2-\alpha^2-2\alpha\beta-\beta^2\ge \beta(c_2 - 2\alpha - \beta) \ge \frac{c_2}{2}\beta,
\end{equation}
for all $\x \in {\bf bd}(B)$ with $\norm{\x - \x^*}\le \epsilon_1$.

\bibliographystyle{abbre}

\begin{thebibliography}{10}

        \bibitem{aragon2008characterization}
        F.~J. Arag{\'o}n~Artacho and M.~H. Geoffroy.
        \newblock Characterization of metric regularity of subdifferentials.
        \newblock {\em Journal of Convex Analysis}, 15(2):365--380, 2008.

        \bibitem{attouch2010proximal}
        H.~Attouch, J.~Bolte, P.~Redont, and A.~Soubeyran.
        \newblock Proximal alternating minimization and projection methods for
        nonconvex problems: An approach based on the {{Kurdyka-{\L}ojasiewicz}}
        inequality.
        \newblock {\em Mathematics of Operations Research}, 35(2):438--457, 2010.

        \bibitem{beck2017first}
        A.~Beck.
        \newblock {\em First-order Methods in Optimization}, volume~25.
        \newblock SIAM, Philadelphia, 2017.

        \bibitem{beck2018globally}
        A.~Beck and Y.~Vaisbourd.
        \newblock Globally solving the trust region subproblem using simple first-order
        methods.
        \newblock {\em SIAM Journal on Optimization}, 28(3):1951--1967, 2018.

        \bibitem{ben2009robust}
        A.~Ben-Tal, L.~El~Ghaoui, and A.~Nemirovski.
        \newblock {\em Robust Optimization}.
        \newblock {Princeton University Press}, 2009.

        \bibitem{bolte2017error}
        J.~Bolte, T.~P. Nguyen, J.~Peypouquet, and B.~W. Suter.
        \newblock From error bounds to the complexity of first-order descent methods
        for convex functions.
        \newblock {\em Mathematical Programming}, 165(2):471--507, 2017.

        \bibitem{burke1993weak}
        J.~V. Burke and M.~C. Ferris.
        \newblock Weak sharp minima in mathematical programming.
        \newblock {\em SIAM Journal on Control and Optimization}, 31(5):1340--1359,
        1993.

        \bibitem{carmon2018analysis}
        Y.~Carmon and J.~C. Duchi.
        \newblock Analysis of Krylov subspace solutions of regularized non-convex
        quadratic problems.
        \newblock In {\em Advances in Neural Information Processing Systems}, pages
        10705--10715, 2018.

        \bibitem{conn2000trust}
        A.~R. Conn, N.~I. Gould, and P.~L. Toint.
        \newblock {\em Trust Region Methods}, volume~1.
        \newblock {SIAM, Philadelphia}, 2000.

        \bibitem{cui2019r}
        Y.~Cui, D.~Sun, and K.-C. Toh.
        \newblock On the {{R-superlinear}} convergence of the {{KKT}} residuals
        generated by the augmented lagrangian method for convex composite conic
        programming.
        \newblock {\em Mathematical Programming}, 178(1-2):381--415, 2019.

        \bibitem{dontchev2009implicit}
        A.~L. Dontchev and R.~T. Rockafellar.
        \newblock {\em Implicit Functions and Solution Mappings}.
        \newblock {Springer Monographs in Mathematics, Springer}, 208, 2009.

        \bibitem{drusvyatskiy2018error}
        D.~Drusvyatskiy and A.~S. Lewis.
        \newblock Error bounds, quadratic growth, and linear convergence of proximal
        methods.
        \newblock {\em Mathematics of Operations Research}, 43(3):919--948, 2018.

        \bibitem{drusvyatskiy2014second}
        D. Drusvyatskiy, B.~S. Mordukhovich, T.~T.~A. Nghia.
        \newblock Second-order growth, tilt stability, and metric regularity of the subdifferential.
        \newblock {\em Journal of Convex
        Analysis}, 21:1165--1192, 2014.


        \bibitem{flippo1996duality}
        O.~E. Flippo and B.~Jansen.
        \newblock Duality and sensitivity in nonconvex quadratic optimization over an
        ellipsoid.
        \newblock {\em European journal of operational research}, 94(1):167--178, 1996.

        \bibitem{fortin2004trust}
        C.~Fortin and H.~Wolkowicz.
        \newblock The trust region subproblem and semidefinite programming.
        \newblock {\em Optimization methods and software}, 19(1):41--67, 2004.

        \bibitem{frankel2015splitting}
        P.~Frankel, G.~Garrigos, and J.~Peypouquet.
        \newblock Splitting methods with variable metric for
        {{Kurdyka--{\L}ojasiewicz}} functions and general convergence rates.
        \newblock {\em Journal of Optimization Theory and Applications},
        165(3):874--900, 2015.

        \bibitem{gao2016ojasiewicz}
        B.~Gao, X.~Liu, X.~Chen, and Y.-X. Yuan.
        \newblock On the {{\L}ojasiewicz} exponent of the quadratic sphere
        constrained optimization problem.
        \newblock {\em arXiv preprint arXiv:1611.08781}, 2016.

        \bibitem{gould1999solving}
        N.~I. Gould, S.~Lucidi, M.~Roma, and P.~L. Toint.
        \newblock Solving the trust-region subproblem using the Lanczos method.
        \newblock {\em SIAM Journal on Optimization}, 9(2):504--525, 1999.

        \bibitem{han2017linear}
        D.~Han, D.~Sun, and L.~Zhang.
        \newblock Linear rate convergence of the alternating direction method of
        multipliers for convex composite programming.
        \newblock {\em Mathematics of Operations Research}, 43(2):622--637, 2017.

        \bibitem{hazan2016linear}
        E.~Hazan and T.~Koren.
        \newblock A linear-time algorithm for trust region problems.
        \newblock {\em Mathematical Programming}, 158(1-2):363--381, 2016.

        \bibitem{ho2017second}
        N.~Ho-Nguyen and F.~Kilinc-Karzan.
        \newblock A second-order cone based approach for solving the trust-region
        subproblem and its variants.
        \newblock {\em SIAM Journal on Optimization}, 27(3):1485--1512, 2017.

        \bibitem{jiang2018linear}
        R.~Jiang and D.~Li.
        \newblock A linear-time algorithm for generalized trust region problems.
        \newblock {\em arXiv preprint arXiv:1807.07563}, 2018.

        \bibitem{jiang2019novel}
        R.~Jiang and D.~Li.
        \newblock Novel reformulations and efficient algorithms for the generalized
        trust region subproblem.
        \newblock {\em SIAM Journal on Optimization}, 29(2):1603--1633, 2019.

        \bibitem{jiang2018socp}
        R.~Jiang, D.~Li, and B.~Wu.
        \newblock {SOCP} reformulation for the generalized trust region subproblem via
        a canonical form of two symmetric matrices.
        \newblock {\em Mathematical Programming}, 169(2):531--563, 2018.

        \bibitem{li2018calculus}
        G.~Li and T.~K. Pong.
        \newblock Calculus of the exponent of Kurdyka--{\L}ojasiewicz inequality and
        its applications to linear convergence of first-order methods.
        \newblock {\em Foundations of computational mathematics}, 18(5):1199--1232,
        2018.

        \bibitem{li1995error}
        W.~Li.
        \newblock Error bounds for piecewise convex quadratic programs and
        applications.
        \newblock {\em SIAM Journal on Control and Optimization}, 33(5):1510--1529,
        1995.

        \bibitem{liu2016quadratic}
        H.~Liu, W.~Wu, and A.~M.-C. So.
        \newblock Quadratic optimization with orthogonality constraints: Explicit
        {\L}ojasiewicz exponent and linear convergence of line-search methods.
        \newblock In {\em ICML}, pages 1158--1167, 2016.

        \bibitem{luo2000error}
        Z.-Q. Luo and J.~F. Sturm.
        \newblock Error bounds for quadratic systems.
        \newblock In {\em High performance optimization}, pages 383--404. Springer,
        2000.

        \bibitem{luo1993error}
        Z.-Q. Luo and P.~Tseng.
        \newblock Error bounds and convergence analysis of feasible descent methods: a
        general approach.
        \newblock {\em Annals of Operations Research}, 46(1):157--178, 1993.

        \bibitem{martinez1994local}
        J.~M. Mart{\'\i}nez.
        \newblock Local minimizers of quadratic functions on {{Euclidean}} balls and
        spheres.
        \newblock {\em SIAM Journal on Optimization}, 4(1):159--176, 1994.

        \bibitem{more1993generalizations}
        J.~J. Mor{\'e}.
        \newblock Generalizations of the trust region problem.
        \newblock {\em Optimization Methods and Software}, 2(3-4):189--209, 1993.

        \bibitem{more1983computing}
        J.~J. Mor{\'e} and D.~C. Sorensen.
        \newblock Computing a trust region step.
        \newblock {\em SIAM Journal on Scientific and Statistical Computing},
        4(3):553--572, 1983.

        \bibitem{nesterov1998introductory}
        Y.~Nesterov.
        \newblock {\em Introductory Lectures on Convex Programming: a Basic Course}. Kluwer,
        Boston, 2004.

        \bibitem{pang1997error}
        J.-S. Pang.
        \newblock Error bounds in mathematical programming.
        \newblock {\em Mathematical Programming}, 79(1-3):299--332, 1997.


        \bibitem{poliquin2014generalized}
        R. Poliquin and R.~T. Rockafellar.
        \newblock Prox-regular functions in variational analysis.
        \newblock {\em Transactions of the American Mathematical Society}, 348(5):1805--1838, 1996.

        \bibitem{pong2014generalized}
        T.~K. Pong and H.~Wolkowicz.
        \newblock The generalized trust region subproblem.
        \newblock {\em Computational Optimization and Applications}, 58(2):273--322,
        2014.

        \bibitem{rendl1997semidefinite}
        F.~Rendl and H.~Wolkowicz.
        \newblock A semidefinite framework for trust region subproblems with
        applications to large scale minimization.
        \newblock {\em Mathematical Programming}, 77(1):273--299, 1997.


        \bibitem{wang2017linear}
        J.~Wang and Y.~Xia.
        \newblock A linear-time algorithm for the trust region subproblem based on
        hidden convexity.
        \newblock {\em Optimization Letters}, 11(8):1639--1646, 2017.

        \bibitem{wang1994global}
        T.~Wang and J.-S. Pang.
        \newblock Global error bounds for convex quadratic inequality systems.
        \newblock {\em Optimization}, 31(1):1--12, 1994.

        \bibitem{yakubovich1971s}
        V.~A. Yakubovich.
        \newblock {S-Procedure} in nonlinear control theory.
        \newblock {\em Vestnik Leningrad University}, 1:62--77, 1971.

        \bibitem{yuan2015recent}
        Y.~Yuan.
        \newblock Recent advances in trust region algorithms.
        \newblock {\em Mathematical Programming}, 151(1):249--281, 2015.

        \bibitem{zhang2010derivative}
        H.~Zhang, A.~R. Conn, and K.~Scheinberg.
        \newblock A derivative-free algorithm for least-squares minimization.
        \newblock {\em SIAM Journal on Optimization}, 20(6):3555--3576, 2010.

        \bibitem{zhang2019geometric}
        H.~Zhang, A.~Milzarek, Z.~Wen, and W.~Yin.
        \newblock On the geometric analysis of a quartic-quadratic optimization problem
        under a spherical constraint.
        \newblock {\em arXiv preprint arXiv:1908.00745}, 2019.

        \bibitem{zhang2018nested}
        L.-H. Zhang and C.~Shen.
        \newblock A nested Lanczos method for the trust-region subproblem.
        \newblock {\em SIAM Journal on Scientific Computing}, 40(4):2005--2032, 2018.

        \bibitem{zhang2017generalized}
        L.-H. Zhang, C.~Shen, and R.-C. Li.
        \newblock On the generalized Lanczos trust-region method.
        \newblock {\em SIAM Journal on Optimization}, 27(3):2110--2142, 2017.

\end{thebibliography}

\end{document}